\documentclass[12pt,reqno]{amsart}
\newcommand{\tf}{\tfrac}

\renewcommand{\Im}{\operatorname{Im}}

\renewcommand{\(}{\left\(}
\renewcommand{\)}{\right\)}
\newcommand{\abs}[1]{\lvert#1\rvert}
\newcommand\mylabel[1]{\label{#1}}
\newcommand\eqn[1]{(\ref{eq:#1})}
\newcommand\thm[1]{\ref{thm:#1}}
\newcommand\lem[1]{\ref{lem:#1}}

\newcommand\corol[1]{\ref{cor:#1}}
\newcommand\sect[1]{\ref{sec:#1}}
\newcommand\subsect[1]{\ref{subsec:#1}}
\newcommand\Etwid{\overset {\text{\lower 3pt\hbox{$\sim$}}}E}
\newcommand\thpow[1]{${#1}^{\text{th}}$}
\newcommand\dashsum{\sum_{n=-\infty \atop{n\neq 0}}^{\infty}}
\renewcommand{\pmod}[1]{\,(\textup{mod}\,#1)}
\numberwithin{equation}{section}
 \theoremstyle{plain}
\newtheorem{theorem}{Theorem}[section]
\newtheorem{lemma}[theorem]{Lemma}
\newtheorem{corollary}[theorem]{Corollary}

\newenvironment{example}{\par\noindent{\bf Example}\hskip 0.5em\ignorespaces}
   {\hfill \par}

\begin{document}
\title[Rank-Crank type PDEs]{Rank-Crank type PDEs 
and \\
generalized Lambert series identities}
\author{Song Heng Chan, Atul Dixit and Frank G. Garvan}
\address{Division of Mathematical Sciences, School of Physical and Mathematical Sciences, Nanyang Technological University, 21 Nanyang link, Singapore, 637371, Republic of Singapore}\email{ChanSH@ntu.edu.sg}
\address{Department of Mathematics, University of Illinois, 1409 West Green
Street, Urbana, IL 61801, USA} \email{aadixit2@illinois.edu}
\address{Department of Mathematics, University of Florida, Gainesville, Florida 32611, USA} \email{fgarvan@ufl.edu}
\thanks{The third author was supported in part by NSA Grant H98230-09-1-0051.}
%%\begin{abstract}
%%A new proof of an identity due to S.~H.~Chan involving the level $(2m+1)$ 
%%Appell functions is obtained using the theory of elliptic functions. This gives
%%an identity of A. O. L.~Atkin and H. P. F.~Swinnerton-Dyer and an identity of 
%%M.~Jackson as special cases. Then we show how Zwegers's Rank-Crank type PDE 
%%arises naturally from Chan's identity. This PDE gives the Rank-Crank PDE due 
%%to the third author and Atkin as well as another Rank-Crank type PDE due to 
%%the third author as special cases.
%%\end{abstract}
\begin{abstract}
We show how Rank-Crank type PDEs for higher order Appell functions
due to Zwegers may be obtained from a generalized Lambert series
identity due to the first author. 
Special cases are the Rank-Crank PDE due to Atkin
and the
third author and  a PDE for a level $5$ Appell function also  found by the third
author. These two special PDEs are related to generalized Lambert series identities
due to Watson, and Jackson respectively. The first author's Lambert
series identities are common generalizations. We also show how Atkin and 
Swinnerton-Dyer's proof using  elliptic functions can be extended to prove these
generalized Lambert series identities.
\end{abstract}

%    General info

%%\date{December 31, 2011}
\date{January 8, 2012}

\dedicatory{Dedicated to our friends, Mourad Ismail and Dennis Stanton}


\maketitle
%%SECTION 1%%%%%%%%%%%%%%%%%%%%%%%%%%%%%%%%%%%%%%%%%%%%%%%%%%%%%%%%%%%%%%%%%%
\section{Introduction}
\mylabel{sec:intro}  %%Section 1
F.~J.~Dyson \cite{dys}, \cite[p.~52]{dys1} defined the rank of a
partition as the largest part minus the number of parts. 
Dyson conjectured that the residue of the rank mod $5$ divides the partitions
of $5n+4$ into $5$ equal clases thereby providing a combinatorial
interpretation of Ramanujan's famous partition 
congruences $p(5n+4)\equiv 0\pmod{5}$. He also conjectured that the rank mod $7$
likewise gives Ramanujan's partition congruence $p(7n+5)\equiv0\pmod{7}$.
Dyson's rank conjectures were proved by
A.~O.~L.~Atkin and H.~P.~F.~Swinnerton-Dyer \cite{as}.
The following was the crucial identity that Atkin and Swinnerton-Dyer needed
for the proof of the Dyson rank conjectures. It was first proved by
G.N.~Watson \cite{Wa}.
\begin{align}       
& \zeta\frac{[\zeta^2]_{\infty}}{[\zeta]_{\infty}}
\sum_{n=-\infty}^{\infty} \frac{(-1)^nq^{3n(n+1)/2}}
                               {1-zq^n}
+\frac{[\zeta]_{\infty}[\zeta^2]_{\infty}(q)_{\infty}^{2}}
      {[z/\zeta]_{\infty}[z]_{\infty}[\zeta z]_{\infty}} 
\mylabel{eq:asi} \\
&\qquad =\sum_{n=-\infty}^{\infty}(-1)^{n}q^{3n(n+1)/2}
  \left(\frac{\zeta^{-3n}}
             {1-zq^{n}/\zeta}
                              +\frac{\zeta^{3n+3}}
                                    {1-z\zeta q^n}\right).
\nonumber
\end{align}        
Throughout we use the standard $q$-notation
%%%%%%%%%%%%%%
\begin{align}
(x)_{0}&:=(x;q)_{0}:=1,\nonumber\\
(x)_{n}&:=(x;q)_{n}:=\prod_{m=0}^{n-1}(1-xq^{m}),\nonumber\\
(x_{1},\cdots, x_{m})_{n}&:=(x_{1},\cdots, x_{m};q)_{n}:=(x_{1};q)_{n}\cdots (x_{m};q)_{n},\nonumber\\
[x_{1},\cdots, x_{m}]_{n}&:=[x_{1},\cdots, x_{m};q]_{n}:=(x_{1}, q/x_{1},\cdots, x_{m}, q/x_{m};q)_{n}.\nonumber
%(x)_{\infty}:=\prod_{m=0}^{infty}(1-xq^{m}),\nonumber
\end{align}
when $n$ is a nonnegative integer. Assuming $\abs{q}<1$ we also
use this notation when $n=\infty$ by interpreting its meaning as the limit
as $n\to\infty$.
%%% ADD? --- Throughout the paper, $q=e^{2\pi i\tau}$ with Im$(\tau)>0$. 
Later M.~Jackson \cite{jack} proved an analogue of the above identity,
%{\allowdisplaybreaks\begin{align}\mylabel{jackson}
%&\frac{\zeta^{2}(\zeta^{2})_{\infty}(q/\zeta^{2})_{\infty}(x\zeta)_{\infty}(q/(x\zeta))_{\infty}(x/\zeta)_{\infty}(q\zeta/x)_{\infty}}{(\zeta)_{\infty}(q/\zeta)_{\infty}((x)_{\infty}(q/x)_{\infty})^2}\sum_{n=-\infty}^{\infty}\frac{(-1)^nq^{5n(n+1)/2}}{1-zq^n}\nonumber\\
%&+\frac{(\zeta)_{\infty}(q/\zeta)_{\infty}(\zeta^{2})_{\infty}(q/\zeta^{2})_{\infty}(x\zeta)_{\infty}(q/(x\zeta))_{\infty}(\zeta/x)_{\infty}(qx/\zeta)_{\infty}}{(z/\zeta)_{\infty}(q\zeta/z)_{\infty}(z)_{\infty}(q/z)_{\infty}(z\zeta)_{\infty}(q/(z\zeta))_{\infty}(z/x)_{\infty}(qx/z)_{\infty}(zx)_{\infty}(q/(zx))_{\infty}}\nonumber\\
%&+\frac{\zeta}{x}\frac{(\zeta)_{\infty}(q/\zeta)_{\infty}(\zeta^{2})_{\infty}(q/\zeta^{2})_{\infty}}{(x)_{\infty}(q/x)_{\infty}(x^{2})_{\infty}(q/x^{2})_{\infty}}\sum_{n=-\infty}^{\infty}(-1)^{n}q^{5n(n+1)/2}\left(\frac{x^{-3n}}{1-zq^{n}/x}+\frac{x^{3n+3}}{1-zx q^n}\right)\nonumber\\
%&=\sum_{n=-\infty}^{\infty}(-1)^{n}q^{5n(n+1)/2}\left(\frac{\zeta^{-5n}}{1-zq^{n}/\zeta}+\frac{\zeta^{5n+5}}{1-z\zeta q^n}\right).
%\end{align}}
\begin{align}%%\mylabel{jackson}
&\frac{\zeta^{2}[\zeta^2]_{\infty}[x\zeta]_{\infty}[x/\zeta]_{\infty}}{[\zeta]_{\infty}[x]_{\infty}^2}\sum_{n=-\infty}^{\infty}\frac{(-1)^nq^{5n(n+1)/2}}{1-zq^n}+\frac{[\zeta]_{\infty}[\zeta^2]_{\infty}[x\zeta]_{\infty}[\zeta/x]_{\infty}(q)_{\infty}^{2}}{[z/x]_{\infty}[z/\zeta]_{\infty}[z]_{\infty}[z\zeta]_{\infty}[zx]_{\infty}}\nonumber\\
&+\frac{\zeta}{x}\frac{[\zeta]_{\infty}[\zeta^2]_{\infty}}{[x]_{\infty}[x^2]_{\infty}}\sum_{n=-\infty}^{\infty}(-1)^{n}q^{5n(n+1)/2}\left(\frac{x^{-5n}}{1-zq^{n}/x}+\frac{x^{5n+5}}{1-zx q^n}\right)
\mylabel{eq:jackson}\\
&=\sum_{n=-\infty}^{\infty}(-1)^{n}q^{5n(n+1)/2}\left(\frac{\zeta^{-5n}}{1-z\zeta^{-1}q^{n}}+\frac{\zeta^{5n+5}}{1-z\zeta q^n}\right).
\nonumber
\end{align}
Recently, 
the first author \cite[p.603]{chan} found a generalization of the above 
two identities, 
namely,
\begin{align}%%\mylabel{eq:shc}
&\frac{x_{1}^m[x_{2}/x_{1},\cdots, x_{m}/x_{1}, x_{1}x_{m},\cdots, x_{1}x_{2}, x_{1}^{2}]_{\infty}}{[x_{1}]_{\infty}[x_{2},\cdots, x_{m}]_{\infty}^{2}}\sum_{n=-\infty}^{\infty}\frac{(-1)^nq^{(2m+1)n(n+1)/2}}{1-zq^n}\nonumber\\
&+\frac{[x_{1}/x_{2},\cdots, x_{1}/x_{m}, x_{1}, x_{1}x_{m},\cdots, x_{1}x_{2}, x_{1}^{2}]_{\infty}(q)_{\infty}^{2}}{[z/x_{1}, z/x_{2},\cdots, z/x_{m}, z, zx_{m},\cdots, zx_{1}]_{\infty}}\nonumber\\
&+\bigg\{\frac{x_{1}}{x_{2}}\frac{[x_{1}/x_{3},\cdots, x_{1}/x_{m}, x_{1}, x_{1}x_{m},\cdots, x_{1}x_{3}, x_{1}^{2}]_{\infty}}{[x_{2}/x_{3},\cdots, x_{2}/x_{m}, x_{2}, x_{2}x_{m},\cdots, x_{2}^{2}]_{\infty}}
\mylabel{eq:shc}\\
&\quad\quad\times\sum_{n=-\infty}^{\infty}(-1)^{n}q^{(2m+1)n(n+1)/2}\left(\frac{x_{2}^{-(2m+1)n}}{1-zq^{n}/x_{2}}+\frac{x_{2}^{(2m+1)(n+1)}}{1-zx_{2} q^n}\right)+\text{idem}(x_{2}; x_{3},\cdots, x_{m})\bigg\}\nonumber\\
&\quad=\sum_{n=-\infty}^{\infty}(-1)^{n}q^{(2m+1)n(n+1)/2}\left(\frac{x_{1}^{-(2m+1)n}}{1-zq^{n}/x_{1}}+\frac{x_{1}^{(2m+1)(n+1)}}{1-zx_{1} q^n}\right),
\nonumber
\end{align}
where $g(a_{1},a_{2},\cdots, a_{m})+$ idem$(a_{1};a_{2},\cdots, a_{n})$ denotes the sum 
$$
\sum_{i=1}^{n}g(a_{i},a_{2},\cdots, a_{i-1}, a_{1},a_{i+1},\cdots, a_{m}),
$$ 
in which  the $i$-th term of the sum is obtained from the first by interchanging 
$a_{1}$ and $a_{i}$.

Equation \eqn{shc} was proved using partial fractions. 
Indeed, the $m=1$ case of \eqn{shc} is equivalent to \eqn{asi}, while 
the $m=2$ case is equivalent to \eqn{jackson}. The fact that the right-hand 
side of \eqn{jackson} is independent of $x$, and that the right-hand 
side of \eqn{shc} is independent of $x_{2}, x_{3}, \cdots, x_{m}$ seems to be 
intriguing at first. Indeed, one purpose of this article is to show that 
the left-hand sides of \eqn{jackson} and \eqn{shc} are really elliptic 
functions of order less than $2$, in fact entire functions as we show, in the 
respective variables ($x$ for \eqn{jackson} and $x_{2}$ for \eqn{shc} 
while holding $x_{3},\cdots,x_{m}$ fixed) and therefore that they must be constants 
which are nothing but the right-hand sides of \eqn{jackson} and \eqn{shc} 
respectively. Since \eqn{jackson} follows from \eqn{shc}, we show this 
only for \eqn{shc}. This is done in Section \sect{chansidelliptic}.

Let $N(m,n)$ denote the number of partitions of $n$ with rank $m$. Then the rank 
generating function $R(z,q)$ is given by
\begin{equation}
R(z,q)=\sum_{n=0}^{\infty}\sum_{m=-n}^n N(m,n)z^mq^n
      =\sum_{n=0}^{\infty}\frac{q^{n^2}}
                               {(zq)_{n}(z^{-1}q)_{n}}.
\mylabel{rgf}
\end{equation} 
In \cite{ang}, G.~E.~Andrews and the third author defined the crank of a partition, a partition 
statistic hypothesized by Dyson in \cite{dys}. It is the largest part if the partition 
contains no ones, and otherwise is the number of parts larger than the number of 
ones minus the number of ones. For $n>1$, we let $M(m,n)$ denote the number of partitions of $n$ 
with crank $m$. 
If we amend the definition of $M(m,n)$ for $n=1$, then
the generating function can be given as an infinite product.
Accordingly, we assume
$$
M(0,1)=-1,\, M(-1,1)=M(1,1)=1,\, \mbox{and $M(m,1)=0$ otherwise}.
$$
Then the crank generating function $C(z,q)$ is given by
\begin{equation}
C(z,q)=\sum_{n=0}^{\infty}\sum_{m=-n}^n M(m,n)z^mq^n
=\frac{(q)_{\infty}}
      {(zq)_{\infty}(z^{-1}q)_{\infty}}.
\mylabel{cgf}
\end{equation}
Atkin and the third author \cite{ag} found the so-called Rank-Crank PDE, 
a partial differential equation (PDE) which relates $R(z,q)$ and $C(z,q)$.
To state this PDE in its original form, we first define the differential operators 
\begin{equation}
\delta_{z}=z\frac{\partial}{\partial z},\hspace{5mm} 
\delta_{q}=q\frac{\partial}{\partial q}.
\mylabel{deltaop}
\end{equation}
Then the Rank-Crank PDE can be written as
\begin{align}\mylabel{eq:rcpde}
z(q)_{\infty}^{2}[C^{*}(z,q)]^{3}=\left(3\delta_{q}+\frac{1}{2}\delta_{z}+\frac{1}{2}\delta_{z}^{2}\right)R^{*}(z,q),
\end{align}
where 
\begin{align}\mylabel{mrcgf}
R^{*}(z,q)&:=\frac{R(z,q)}{1-z},\nonumber\\
C^{*}(z,q)&:=\frac{C(z,q)}{1-z}.
\end{align}
In \cite{ag}, it was shown how the Rank-Crank PDE and certain results for the 
derivatives of 
Eisenstein series lead to exact relations between rank and crank moments.
%% which 
%%are defined by, 
%%respectively,
%%\begin{align}
%%N_{j}(n)&=\sum_{k}k^{j}N(k,n),\mylabel{Njmom}\\
%%M_{j}(n)&=\sum_{k}k^{j}M(k,n).\mylabel{Mjmom}
%%\end{align}
As in \cite{gmm}, define $N_k(m,n)$ by
\begin{equation}
\sum_{n\ge0} N_k(m,n) q^n 
= \frac{1}{(q)_\infty} \sum_{n=1}^\infty (-1)^{n-1} q^{n((2k-1)n-1)/2 + \abs{m}n}
(1-q^n),
\mylabel{Nkdef}
\end{equation}
of any positive integer $k$. When $k=1$ this is the generating function for the
crank, and when $k=2$ it is the generating function for the rank. When $k\ge2$,
$N_k(m,n)$ can be interpreted combinatorially as the number of partitions of $n$
into $k-1$ successive Durfee squares with $k$-rank equal to $m$. See
\cite[Eq.(1.11)]{gmm} for a definition of the $k$-rank. We define
\begin{equation}
R_k(z,q) := \sum_{n\ge0} \sum_{m=-n}^n N_k(m,n) z^m q^n.
\mylabel{Rkdef}
\end{equation}
From \cite[Eq.(4.5)]{gmm}, this generating function can be written as
\begin{equation}\mylabel{r3gfk}
R_k(z,q)=
\sum_{n_{k-1}\geq n_{k-2}\geq\cdots\geq n_{1}\geq 1}
\frac{q^{n_{1}^{2}+n_{2}^{2}+\cdots +n_{k-1}^{2}}}
{(q)_{n_{k-1}-n_{k-2}}\cdots(q)_{n_{2}-n_{1}}(zq)_{n_{1}}(z^{-1}q)_{n_{1}}},
\end{equation}
when $k\ge2$.
In Section \sect{RkSigproof}, we show that $R_k(z,q)$ is related to the level 
$2k-1$ Appell function
\begin{equation}
\Sigma^{(2k-1)}(z,q) := 
\sum_{n=-\infty}^{\infty}\frac{(-1)^{n}q^{(2k-1)n(n+1)/2}}{1-zq^{n}}.
\end{equation}
We obtain the following
\begin{theorem}
\mylabel{thm:RkSig}
For $k\ge1$,
\begin{align}
&R_k(z,q) \nonumber\\
&=\frac{1}{(q)_\infty}\left(
z^{k-1}(1-z)\Sigma^{(2k-1)}(z,q)-z\theta_{1,2k-1}(q)+
z(1-z)\sum_{m=0}^{k-3}z^{m}\theta_{2m+3,2k-1}(q)
\right),
\mylabel{eq:RkID}
\end{align}
where
\begin{equation}
\theta_{j,2k-1}(q)=\sum_{n=-\infty}^{\infty}(-1)^{n}q^{n((2k-1)n+j)/2},
\mylabel{eq:thetajkdef}
%% changed "-j" to "+j" (better)
\end{equation}
for $j=1$, $3$, \dots, $2k-3$.
\end{theorem}
This theorem generalizes Lemma 7.9 in \cite{gar} 
which gives a relation between the rank generating function $R(z,q)$ and a 
level $3$ Appell function.
The $k=1$ case of the theorem gives the familiar partial fraction expansion for the 
reciprocal of Jacobi's theta product $(z)_{\infty}(z^{-1}q)_{\infty}$ 
\cite[p.~1]{lnb}, \cite[p.~136]{tm}, 
since $\sum_{n=-\infty}^{\infty}(-1)^nq^{n(n+1)/2}=0$.

A few years ago 
the third author found a $4^{\text{th}}$ order PDE, 
which is an analogue of the Rank-Crank PDE and is related to the 
$3$-rank \cite{gmm}. To state this PDE we define
\begin{equation}
G^{(5)}(z,q) := \frac{1}{(q)_\infty^3} \Sigma^{(5)}(z,q).
\mylabel{eq:G5def}
\end{equation}
Then
\begin{align}
&24(q)_{\infty}^{2}[C^{*}(z,q)]^{5} \nonumber\\
&=24(1-10\Phi_3(q))G^{(5)}(z,q)\nonumber\\
&\quad+\left(100\delta_{q}+50\delta_{z}+100\delta_{q}\delta_{z}
+35\delta_{z}^{2}+20\delta_{q}\delta_{z}^{2}
+100\delta_{q}^{2}+10\delta_{z}^{3}+\delta_{z}^{4}\right)
G^{(5)}(z,q),
\mylabel{eq:grcpde}
\end{align}
where
\begin{equation}
\Phi_{3}(q):=\sum_{n=1}^{\infty}\frac{n^3q^n}{1-q^{n}}.
\mylabel{eq:E4twidDEF}
\end{equation}
This PDE can be written more compactly as
\begin{equation}
24(q)_{\infty}^{2}[C^{*}(z,q)]^{5} = (\mathbf{H}_{*}^2 - E_4)\, G^{(5)}(z,q), 
\mylabel{eq:grcpdev2}
\end{equation}
where  $\mathbf{H}_{*}$ is the operator
$$
\mathbf{H}_{*} := 5 + 10 \delta_q + 5 \delta_z + \delta_z^2, 
$$
and
$$
E_4 := E_4(q) := 1 + 240 \Phi_{3}(q),
$$
is the usual Eisenstein series of weight $4$.
%%pp:=100*dq+50*dz+100*dq*dz+35*dz^2+20*dq*dz^2+100*dq^2 + 10*dz^3+dz^4;
%%                                    2           2         2        3     4
%%  100 dq + 50 dz + 100 dq dz + 35 dz  + 20 dq dz  + 100 dq  + 10 dz  + dz 
%%factor(pp);
%%               /                      2\ /                 2\
%%               \10 dq + 10 + 5 dz + dz / \10 dq + 5 dz + dz /
%%H5:=dz^2+5*dz+10*dq;
%%                                              2
%%                             10 dq + 5 dz + dz 
%%expand(H5*(H5+10)-pp);
%%                                      0
The PDE \eqn{grcpde} was first conjectured by the third author and 
then subsequently proved 
and generalized by Zwegers \cite{zwe}. It was also Zwegers who first observed
that \eqn{grcpde} could be written in a more compact form. 
We now describe Zwegers's  generalization.
%%%%%%%
%%%%%%%
Define for $l\in\mathbb{Z}_{>0}$, the level $l$ Appell function as
\begin{equation}\mylabel{appell}
A_{l}(u, v):=A_{l}(u, v;\tau):=z^{l/2}\sum_{n=-\infty}^{\infty}\frac{(-1)^{l n}q^{ln(n+1)/2}w^{n}}{1-zq^{n}},
\end{equation}
where $z=e^{2\pi iu}$, $w=e^{2\pi iv}$, $q=e^{2\pi i\tau}$, and 
define the modified rank and crank generating functions as follows.
\begin{align}
\mathcal{R}&:=\mathcal{R}(u;\tau):=\frac{z^{1/2}q^{-1/24}}{1-z}R(z,q),
\mylabel{eq:ZR}\\
\mathcal{C}&:=\mathcal{C}(u;\tau):=\frac{z^{1/2}q^{-1/24}}{1-z}C(z,q).
\mylabel{eq:ZC}
\end{align}
Here and throughout we assume $\Im\tau>0$ so that $\abs{q}<1$.
Then the following theorem due to Zwegers gives for odd $l$, 
the $(l-1)^{\text{th}}$ order analogue of the Rank-Crank PDE.
\begin{theorem}[Zwegers\cite{zwe}]
\mylabel{thm:zwegers}
Let $l\geq 3$ be an odd integer. Define
\begin{align}
\mathcal{H}_{k}&:=\frac{l}{\pi i}\frac{\partial}{\partial\tau}+\frac{1}{(2\pi i)^{2}}\frac{\partial^{2}}{\partial u^{2}}-\frac{l(2k-1)}{12}E_{2},\nonumber\\
\mathcal{H}^{[k]}&:=\mathcal{H}_{2k-1}\mathcal{H}_{2k-3}\cdots\mathcal{H}_{3}\mathcal{H}_{1},
\end{align} 
where $E_{2}(\tau)=1-24\sum_{n=1}^{\infty}\sigma_{1}(n)q^{n}$ is the usual 
Eisenstein series of weight $2$ with $\sigma_{\alpha}(n)=\sum_{d|n}d^{\alpha}$.
Then there exist holomorphic modular forms $f_{j}$ (which can be constructed 
explicitly), with $j=4,6,8,\cdots,l-1$, on $SL_{2}(\mathbb{Z})$ of weight 
$j$, such that
\begin{equation}
\left(\mathcal{H}^{[(l-1)/2]}+\sum_{k=0}^{(l-5)/2}f_{l-2k-1}\mathcal{H}^{[k]}\right)A_{l}(u,0)=(l-1)!\eta^{l}\mathcal{C}^{l},
\mylabel{eq:ZwegersPDE}
\end{equation}
where $\eta$ is the Dedekind $\eta$-function, given by 
$\eta(\tau)=q^{1/24}(q)_{\infty}$.
\end{theorem}
%%Note that with $l=2m+1$, $l$ odd and $l\geq 3$ implies $m\geq 1$.\\

Zwegers proved this theorem using the formulas and methods motivated from the 
theory of Jacobi forms. In contrast to this, the proof of the Rank-Crank 
PDE by Atkin and the third author, which corresponds to the $l=3$ case 
of Zwegers's PDE, depends upon simply taking the second derivative with 
respect to $\zeta$ of both sides of \eqn{asi}. 
The main goal of this paper is to show how
a generalized Rank-Crank PDE of any odd order
follows from the Lambert series identity \eqn{shc} in a similar fashion.
We will obtain Zwegers's result in a different form. 
In our form the
coefficients are quasimodular forms rather than holomorphic modular
forms, but in contrast, our coefficients are given recursively.
See Theorem \thm{maintheorem} and Corollary \corol{maincor}.

This paper is organized as follows. In Section \sect{chansidelliptic}, 
we prove \eqn{shc} using the theory of elliptic functions. 
Then in Section \sect{RkSigproof} we prove Theorem \thm{RkSig}, which is
the theorem that relates $R_k(z,q)$ with the level $2k-1$ Appell function
$\Sigma^{(2k-1)}(z,q)$. %%HERE
In Section \sect{higher} we prove our main result that shows how
\eqn{shc} can be used to derive the higher order Rank-Crank-type PDEs of Zwegers.

In the light of (\ref{appell}), it should be observed that the identities 
\eqn{asi} and \eqn{jackson} are really the identities involving certain combinations of level $3$ and level $5$ Appell functions respectively while \eqn{shc} is an identity involving a combination of level $(2m+1)$ Appell functions. However, the analogue for level $1$ Appell function which cannot be derived from \eqn{shc} was found by R.~Lewis \cite[Equation 11]{lewis} and is as follows.
\begin{equation}\mylabel{lew}
\frac{[z]_{\infty}[\zeta^{2}]_{\infty}(q)_{\infty}^{2}}{[z\zeta]_{\infty}[\zeta]_{\infty}[z\zeta^{-1}]_{\infty}}=\sum_{n=-\infty}^{\infty}(-1)^{n}q^{n(n+1)/2}\left(\frac{\zeta^{-n}}{1-zq^{n}/\zeta}+\frac{\zeta^{n+1}}{1-z\zeta q^n}\right).
\end{equation}

%%SECTION 2%%%%%%%%%%%%%%%%%%%%%%%%%%%%%%%%%%%%%%%%%%%%%%%%%%%%%%%%%%%%%%%%%%
\section{General Lambert series identity through elliptic function theory}
\mylabel{sec:chansidelliptic} %% Section 2
Atkin and Swinnerton-Dyer's proof of \eqn{asi} depends in essence on
the theory of elliptic functions.
In this section, we show how this method of proof can be used to prove \eqn{shc}.
Let
\begin{equation}\mylabel{uvw}
z=e^{2\pi iu}, x_{1}=e^{2\pi iv}, x_{2}=e^{2\pi iw}, 
\end{equation}
and let
\begin{equation}\mylabel{a3am}
x_{j}=e^{2\pi ia_{j}}, \hspace{5mm}j=3,\cdots, m,
\end{equation}
where $u, v, w, a_{3}, \dots, a_{m}$ are all complex numbers. 
Also recall that $q=e^{2\pi i\tau}$, where Im $\tau>0$.

Let
\begin{equation}\mylabel{j}
[b]_{\infty}=:J(a,q)=:J(a),\hspace{2mm}\text{where}\hspace{2mm}b=e^{2\pi ia},a\in\mathbb{C}.
\end{equation}
Then using the Jacobi triple product identity \cite[p.~10, Theorem 1.3.3]{ntsr}, we easily find that,
\begin{equation}\mylabel{jtheta}
J(a,q)=\frac{ie^{\pi ia}q^{-1/8}}{(q)_{\infty}}\theta(a),
\end{equation}
where
\begin{equation}\mylabel{deftheta}
\theta(z)=\theta(z;\tau):=\sum_{n=-\infty}^{\infty}e^{\pi i\left(n+\tf{1}{2}\right)^2\tau+2\pi i\left(n+\tf{1}{2}\right)\left(z+\tf{1}{2}\right)}.
\end{equation}
Comparing this with the classical definition of $\theta_{1}(z)$ \cite[p.~355, Section 13.19, Equation 10]{htf}, we find that upon replacing $q$ by $q^{1/2}$ in this classical definition, $\theta(z)=-\theta_{1}(z)$. From \cite[p.~8]{zwegers}, we see that 
\begin{align}
\theta(z+1)&=-\theta(z),\mylabel{thetaqp1}\\
\theta(z+\tau)&=-e^{-\pi i\tau-2\pi iz}\theta(z),\mylabel{thetaqp2}\\
\theta(-z)&=-\theta(z),\mylabel{todd}\\
\theta^{'}(0;\tau)&=-2\pi q^{1/8}(q)_{\infty}^{3}\mylabel{thetaqp3}.
\end{align}
Using (\ref{thetaqp1}) and (\ref{thetaqp2}), we have
\begin{align}
J(a+1,q)&=J(a,q),\mylabel{propj1}\\
J(a+\tau,q)&=-e^{-2\pi ia}J(a,q),\mylabel{propj2}\\
J(a-\tau,q)&=-q^{-1}e^{2\pi ia}J(a,q),\mylabel{propj3}\\
J(-a,q)&=-e^{-2\pi ia}J(a,q).\mylabel{jn}
\end{align}

Using (\ref{j}), we rephrase \eqn{shc} as follows:
\begin{align}\mylabel{shcj}
%g(w)&:=f_{1}(w)+f_{2}(w)+f_{3}(w)\nonumber\\
&\frac{e^{2\pi imv}J(w-v)J(w+v)J(2v)}{J(v)(J(w))^{2}}\displaystyle\prod_{3\leq j\leq m}\frac{J(a_{j}-v)J(a_{j}+v)}{(J(a_{j}))^2}\sum_{n=-\infty}^{\infty}\frac{(-1)^nq^{(2m+1)n(n+1)/2}}{1-e^{2\pi iu}q^n}\nonumber\\
&+\frac{J(v-w)J(v)J(v+w)J(2v)(q)_{\infty}^{2}}{J(u-v)J(u-w)J(u)J(u+w)J(u+v)}\displaystyle\prod_{3\leq j\leq m}\frac{J(v-a_{j})J(v+a_{j})}{J(u-a_{j})J(u+a_{j})}\nonumber\\
&+\bigg\{ e^{2\pi i(v-w)}\frac{J(v)J(2v)}{J(w)J(2w)}\displaystyle\prod_{3\leq j\leq m}\frac{J(v-a_{j})J(v+a_{j})}{J(w-a_{j})J(w+a_{j})}\nonumber\\
&\quad\quad\times\sum_{n=-\infty}^{\infty}(-1)^{n}q^{(2m+1)n(n+1)/2}\left(\frac{e^{-2\pi iw(2m+1)n}}{1-e^{2\pi i(u-w)}q^{n}}+\frac{e^{2\pi iw(2m+1)(n+1)}}{1-e^{2\pi i(u+w)}q^n}\right)+\text{idem}(w; a_{3},\cdots, a_{m})\bigg\}\nonumber\\
&\quad=\sum_{n=-\infty}^{\infty}(-1)^{n}q^{(2m+1)n(n+1)/2}\left(\frac{e^{-2\pi iv(2m+1)n}}{1-e^{2\pi i(u-v)}q^{n}}+\frac{e^{2\pi iv(2m+1)(n+1)}}{1-e^{2\pi i(u+v)}q^n}\right).
\end{align}
Fix $a_{3},\cdots, a_{m}$, consider the left-hand side of (\ref{shcj}) as a function of $w$ only and denote it by $g(w)$. Let $f_{1}(w)$ denote the expression in line 1 of (\ref{shcj}), $f_{2}(w)$ the expression in line 2 of (\ref{shcj}) and $f_{3}(w)$ the expression in lines 3 and 4 of (\ref{shcj}). Then, using (\ref{propj1}), (\ref{propj2}) and (\ref{propj3}), we see                    
\begin{align}\label{f11tau1}
f_{1}(w+1)&=f_{1}(w)=f_{1}(w+\tau),\nonumber\\
f_{2}(w+1)&=f_{2}(w)=f_{2}(w+\tau),\nonumber\\
f_{3}(w+1)&=f_{3}(w).
%Similarly from (\ref{propj1}), (\ref{propj2}) and (\ref{propj3}), we have
%\begin{align}\mylabel{f21}
%f_{2}(w+1)&=f_{2}(w),
%&=f_{2}(w)
%\end{align}
%Now from (\ref{propj2}) and (\ref{propj3}),
%\begin{align}\mylabel{f2tau}
%f_{2}(w+\tau)&=\frac{J(v-w-\tau)J(v)J(v+w+\tau)J(2v)(q)_{\infty}^{2}}{J(u-v)J(u-w-\tau)J(u)J(u+w+\tau)J(u+v)}\displaystyle\prod_{3\leq j\leq m}\frac{J(v-a_{j})J(v+a_{j})}{J(u-a_{j})J(u+a_{j})}\nonumber\\
%&=\frac{\left(-q^{-1}e^{2\pi i(v-w)}J(v-w)\right)J(v)\left(-e^{-2\pi i(v+w)}J(v+w)\right)J(2v)(q)_{\infty}^{2}}{J(u-v)\left(-q^{-1}e^{2\pi i(u-w)}J(u-w)\right)J(u)\left(-e^{-2\pi i(u+w)}J(u+w)\right)J(u+v)}\nonumber\\
%&\quad\times\displaystyle\prod_{3\leq j\leq m}\frac{J(v-a_{j})J(v+a_{j})}{J(u-a_{j})J(u+a_{j})}\nonumber\\
%&=f_{2}(w).
%\end{align}
%Using (\ref{propj1}), we find that
%\begin{equation}\mylabel{f31}
%\end{equation}
\end{align}%
Another application of (\ref{propj2}) and (\ref{propj3}) gives
\begin{align}\mylabel{f3tau}
&f_{3}(w+\tau)\nonumber\\
&=\frac{e^{2\pi i(v-w-\tau)}J(v)J(2v)}{J(w+\tau)J(2w+2\tau)}\displaystyle\prod_{3\leq j\leq m}\frac{J(v-a_{j})J(v+a_{j})}{J(w+\tau-a_{j})J(w+\tau+a_{j})}\nonumber\\
&\quad\times\sum_{n=-\infty}^{\infty}(-1)^{n}q^{(2m+1)n(n+1)/2}\left(\frac{e^{-2\pi i(w+\tau)(2m+1)n}}{1-e^{2\pi i(u-w-\tau)}q^{n}}+\frac{e^{2\pi i(w+\tau)(2m+1)(n+1)}}{1-e^{2\pi i(u+w+\tau)}q^n}\right)\nonumber\\
&\quad+\sum_{k=3}^{m}e^{2\pi i(v-a_{k})}\frac{J(v)J(2v)J(v-w-\tau)J(v+w+\tau)}{J(a_{k})J(2a_{k})J(a_{k}-w-\tau)J(a_{k}+w+\tau)}\displaystyle\prod_{2<j<m+1\atop {j\neq k}}\frac{J(v-a_{j})J(v+a_{j})}{J(a_{k}-a_{j})J(a_{k}+a_{j})}\nonumber\\
&\quad\quad\quad\times\sum_{n=-\infty}^{\infty}(-1)^{n}q^{(2m+1)n(n+1)/2}\left(\frac{e^{-2\pi ia_{k}(2m+1)n}}{1-e^{2\pi i(u-a_{k})}q^{n}}+\frac{e^{2\pi ia_{k}(2m+1)(n+1)}}{1-e^{2\pi i(u+a_{k})}q^n}\right)\nonumber\\
&=e^{2\pi iv+4\pi imw}\frac{J(v)J(2v)}{J(w)J(2w)}\displaystyle\prod_{3\leq j\leq m}\frac{J(v-a_{j})J(v+a_{j})}{J(w-a_{j})J(w+a_{j})}\bigg(\sum_{n=-\infty}^{\infty}(-1)^{n}q^{(2m+1)(n+1)(n+2)/2}\nonumber\\
&\quad\times\frac{e^{-2\pi iw(2m+1)(n+1)}q^{-(2m+1)(n+1)}}{1-e^{2\pi i(u-w)}q^{n}}+\sum_{n=-\infty}^{\infty}(-1)^{n}q^{(2m+1)n(n-1)/2}\frac{e^{2\pi iw(2m+1)n}q^{(2m+1)n}}{1-e^{2\pi i(u+w)}q^{n}}\bigg)\nonumber\\
&\quad+\sum_{k=3}^{m}e^{2\pi i(v-a_{k})}\frac{J(v)J(2v)J(v-w)J(v+w)}{J(a_{k})J(2a_{k})J(a_{k}-w)J(a_{k}+w)}\displaystyle\prod_{2<j<m+1\atop {j\neq k}}\frac{J(v-a_{j})J(v+a_{j})}{J(a_{k}-a_{j})J(a_{k}+a_{j})}\nonumber\\
&\quad\quad\quad\times\sum_{n=-\infty}^{\infty}(-1)^{n}q^{(2m+1)n(n+1)/2}\left(\frac{e^{-2\pi ia_{k}(2m+1)n}}{1-e^{2\pi i(u-a_{k})}q^{n}}+\frac{e^{2\pi ia_{k}(2m+1)(n+1)}}{1-e^{2\pi i(u+a_{k})}q^n}\right)\nonumber\\
&=f_{3}(w).
\end{align}
Thus from (\ref{f11tau1}) and (\ref{f3tau}), we deduce that $g$ is a doubly periodic function in $w$ with periods $1$ and $\tau$. Our next task is to show that $g$ is an entire function of $w$ and hence a constant (with respect to $w$). We show that the poles of $g$ at $w=u$ and $w=-u$ are actually removable singularities by proving that $\lim_{w\to\pm u}(w\mp u)\left(f_{2}(w)+f_{3}(w)\right)=0$ which readily implies that $\lim_{w\to\pm u}(w\mp u)g(w)=0$. Let
\begin{equation}\mylabel{a}
A:=A(v, a_{3},\cdots ,a_{m};q):=J(v)J(2v)\displaystyle\prod_{3\leq j\leq m}J(v-a_{j})J(v+a_{j}).
\end{equation}
%% Removed sum' notation (suggested by SHC) [It was used in only one place]
%%Next, using the notation ${\displaystyle\sum_{n=-\infty}^{\infty}}^{'}$ to 
%%denote the sum $\displaystyle\sum_{n=-\infty}^{\infty}$ with the $n=0$ 
%%term excluded and applying (\ref{jtheta}), (\ref{jn}) and (\ref{thetaqp3}), 
Next, applying (\ref{jtheta}), (\ref{jn}) and (\ref{thetaqp3}), 
we see that
{\allowdisplaybreaks\begin{align}\mylabel{puf2}
&\lim_{w\to u}(w-u)\left(f_{2}(w)+f_{3}(w)\right)\nonumber\\
&=A\lim_{w\to u}(w-u)\bigg\{\frac{J(v-w)J(v+w)(q)_{\infty}^{2}}
  {J(u-w)J(u+w)J(u-v)J(u)J(u+v)}\displaystyle\prod_{3\leq j\leq m}
  \frac{1}{J(u-a_{j})J(u+a_{j})}\nonumber\\
&\quad+\frac{e^{2\pi i(v-w)}}{J(w)J(2w)}\displaystyle\prod_{3\leq j\leq m}\frac{1}{J(w-a_{j})J(w+a_{j})}\bigg(\frac{1}{1-e^{2\pi i(u-w)}}+
%%{\sum_{n=-\infty}^{\infty}}'(-1)^{n}q^{(2m+1)n(n+1)/2}\nonumber\\
\dashsum (-1)^{n}q^{(2m+1)n(n+1)/2}\nonumber\\
&\quad\quad\times \frac{e^{-2\pi iw(2m+1)n}}{1-e^{2\pi i(u-w)}q^{n}}+\sum_{n=-\infty}^{\infty}(-1)^{n}q^{(2m+1)n(n+1)/2}\frac{e^{2\pi iw(2m+1)(n+1)}}{1-e^{2\pi i(u+w)}q^n}\bigg)\bigg\}\nonumber\\
%&=A\bigg(-iq^{1/8}(q)_{\infty}^{3}\frac{J(v-u)J(v+u)}{\left(-e^{-2\pi i(v-u)}J(v-u)\right)J(u)J(2u)J(u+v)}\displaystyle\prod_{3\leq j\leq m}\frac{1}{J(u-a_{j})J(u+a_{j})}\nonumber\\
%&\quad\quad\times \lim_{w\to u}\frac{w-u}{e^{\pi i(u-w)}\theta(u-w)}+\frac{1}{J(u)J(2u)}\displaystyle\prod_{3\leq j\leq m}\frac{1}{J(u-a_{j})J(u+a_{j})}\lim_{w\to u}\frac{(w-u)e^{2\pi i(v-w)}}{1-e^{2\pi i(u-w)}}\bigg)\nonumber\\
&=A\frac{e^{2\pi i(v-u)}}{J(u)J(2u)}\displaystyle\prod_{3\leq j\leq m}\frac{1}{J(u-a_{j})J(u+a_{j})}\left(\frac{-iq^{1/8}(q)_{\infty}^{3}}{\theta^{'}(0)}+\frac{1}{2\pi i}\right)\nonumber\\
&=0.
\end{align}}%%%%
\noindent
Similarly, $\lim_{w\to -u}(w+u)\left(f_{2}(w)+f_{3}(w)\right)=0$.
%Next using the notation ${\displaystyle\sum_{n=-\infty}^{\infty}}^{'}$ to denote the sum $\displaystyle\sum_{n=-\infty}^{\infty}$ with the $n=0$ term excluded, we see that
%{\allowdisplaybreaks\begin{align}\mylabel{puf3}
%&\lim_{w\to u}(w-u)f_{3}(w)\nonumber\\
%&=\lim_{w\to u}e^{2\pi i(v-w)}\frac{(w-u)J(v)J(2v)}{J(w)J(2w)}\displaystyle\prod_{3\leq j\leq m}\frac{J(v-a_{j})J(v+a_{j})}{J(w-a_{j})J(w+a_{j})}\bigg\{\frac{1}{1-e^{2\pi i(u-w)}}\nonumber\\
%&\quad+{\sum_{n=-\infty}^{\infty}}'(-1)^{n}q^{(2m+1)n(n+1)/2}\frac{e^{-2\pi iw(2m+1)n}}{1-e^{2\pi i(u-w)}q^{n}}+\sum_{n=-\infty}^{\infty}(-1)^{n}q^{(2m+1)n(n+1)/2}\frac{e^{2\pi iw(2m+1)(n+1)}}{1-e^{2\pi i(u+w)}q^n}\bigg\}\nonumber\\
%&\quad+\lim_{w\to u}(w-u)\sum_{k=3}^{m}e^{2\pi i(v-a_{k})}\frac{J(v)J(2v)J(v-w)J(v+w)}{J(a_{k})J(2a_{k})J(a_{k}-w)J(a_{k}+w)}\displaystyle\prod_{2<j<m+1\atop {j\neq k}}\frac{J(v-a_{j})J(v+a_{j})}{J(a_{k}-a_{j})J(a_{k}+a_{j})}\nonumber\\
%&\quad\quad\quad\times\sum_{n=-\infty}^{\infty}(-1)^{n}q^{(2m+1)n(n+1)/2}\left(\frac{e^{-2\pi ia_{k}(2m+1)n}}{1-e^{2\pi i(u-a_{k})}q^{n}}+\frac{e^{2\pi ia_{k}(2m+1)(n+1)}}{1-e^{2\pi i(u+a_{k})}q^n}\right)\nonumber\\
%&=\frac{e^{2\pi i(v-u)}}{2\pi i}\frac{J(v)J(2v)}{J(u)J(2u)}\displaystyle\prod_{3\leq j\leq m}\frac{J(v-a_{j})J(v+a_{j})}{J(u-a_{j})J(u+a_{j})},
%\end{align}}
%which is finite. Similarly, $\lim_{w\to -u}(w+u)f_{3}(w)$ is finite.
Now the only other possibility of a pole of $g$ is at $0$, which arises from 
$f_{1}$ and $f_{3}$ each having a pole at $0$. Again, to show that this is a removable 
singularity, it suffices to show that 
$\lim_{w\to 0}w\left(f_{1}(w)+f_{3}(w)\right)=0$. To show this, we need 
Jacobi's duplication formula for theta functions \cite[p.~488, Ex. 5]{ww}
\begin{equation}\mylabel{wwa}
\theta(2w)\theta_{2}\theta_{3}\theta_{4}=2\theta(w)\theta_{2}(w)\theta_{3}(w)\theta_{4}(w).
\end{equation}
Let 
\begin{align}\mylabel{bc}
B&:=B(u, v, a_{3},\cdots, a_{m};q):=e^{2\pi imv}\frac{J(2v)}{J(v)}\displaystyle\prod_{3\leq j\leq m}\frac{J(a_{j}-v)J(a_{j}+v)}{(J(a_{j}))^2}\sum_{n=-\infty}^{\infty}\frac{(-1)^nq^{(2m+1)n(n+1)/2}}{1-e^{2\pi iu}q^n}.\nonumber\\
%C&:=C(v, a_{3},\cdots, a_{m};q):=e^{2\pi iv}J(v)J(2v)\displaystyle\prod_{3\leq j\leq m}J(v-a_{j})J(v+a_{j}).
\end{align}
Then from (\ref{shcj}) and (\ref{bc}),
{\allowdisplaybreaks\begin{align}\mylabel{mc}
&\lim_{w\to 0}w\left(f_{1}(w)+f_{3}(w)\right)\nonumber\\
&=\lim_{w\to 0}w\bigg\{B\frac{J(w-v)J(w+v)}{(J(w))^{2}}+A\frac{e^{2\pi i(v-w)}}{J(w)J(2w)}\displaystyle\prod_{3\leq j\leq m}\frac{1}{J(w-a_{j})J(w+a_{j})}\nonumber\\
&\quad\times\sum_{n=-\infty}^{\infty}(-1)^{n}q^{(2m+1)n(n+1)/2}\left(\frac{e^{-2\pi iw(2m+1)n}}{1-e^{2\pi i(u-w)}q^{n}}+\frac{e^{2\pi iw(2m+1)(n+1)}}{1-e^{2\pi i(u+w)}q^n}\right)\bigg\}\nonumber\\
&\quad+\lim_{w\to 0}w\sum_{k=3}^{m}e^{2\pi i(v-a_{k})}\frac{J(v)J(2v)J(v-w)J(v+w)}{J(a_{k})J(2a_{k})J(a_{k}-w)J(a_{k}+w)}\displaystyle\prod_{2<j<m+1\atop {j\neq k}}\frac{J(v-a_{j})J(v+a_{j})}{J(a_{k}-a_{j})J(a_{k}+a_{j})}\nonumber\\
&=\lim_{w\to 0}\frac{w}{\theta(w)}\bigg\{B\frac{\theta(w-v)\theta(w+v)}{\theta(w)}-Ae^{2\pi iv}q^{1/4}(q)_{\infty}^{2}\frac{e^{-5\pi iw}}{\theta(2w)}\displaystyle\prod_{3\leq j\leq m}\frac{1}{J(w-a_{j})J(w+a_{j})}\nonumber\\
&\quad\times\sum_{n=-\infty}^{\infty}(-1)^{n}q^{(2m+1)n(n+1)/2}\left(\frac{e^{-2\pi iw(2m+1)n}}{1-e^{2\pi i(u-w)}q^{n}}+\frac{e^{2\pi iw(2m+1)(n+1)}}{1-e^{2\pi i(u+w)}q^n}\right)\bigg\}\nonumber\\
&=\frac{1}{\theta^{'}(0)}\lim_{w\to 0}\frac{D(w)}{\theta(w)},
\end{align}}
where 
\begin{align}\mylabel{d}
D(w)&:=D(w, v, a_{3},\cdots, a_{m};q):=B\theta(w-v)\theta(w+v)-Ae^{2\pi iv}q^{1/4}(q)_{\infty}^{2}E(w)\nonumber\\
E(w)&:=E(w; u, a_{3}\cdots, a_{m};q):=\frac{e^{-5\pi iw}\theta(w)}{\theta(2w)}\displaystyle\prod_{3\leq j\leq m}\frac{1}{J(w-a_{j})J(w+a_{j})}\nonumber\\
&\quad\times\sum_{n=-\infty}^{\infty}(-1)^{n}q^{(2m+1)n(n+1)/2}\left(\frac{e^{-2\pi iw(2m+1)n}}{1-e^{2\pi i(u-w)}q^{n}}+\frac{e^{2\pi iw(2m+1)(n+1)}}{1-e^{2\pi i(u+w)}q^n}\right).\nonumber\\
\end{align}
Now using (\ref{wwa}), (\ref{jtheta}) and (\ref{jn}), we find that as $w\to 0$,
{\allowdisplaybreaks\begin{align}\mylabel{dsimp}
D(w)&\to -B\theta^{2}(v)-\frac{e^{2\pi iv}Aq^{1/4}(q)_{\infty}^{2}}{(-1)^{m-2}e^{-2\pi i(a_{3}+\cdots+a_{m})}(J(a_{3})\cdots J(a_{m}))^2}\sum_{n=-\infty}^{\infty}\frac{(-1)^{n}q^{(2m+1)n(n+1)/2}}{1-e^{2\pi iu}q^{n}}\nonumber\\
&=0,
\end{align}}%
which is observed by putting the expressions for $B$ and $C$ back in the first expression on the right side in (\ref{dsimp}). Thus,
\begin{equation}\mylabel{mc2}
\lim_{w\to 0}w\left(f_{1}(w)+f_{3}(w)\right)=\frac{D^{'}(0)}{{\theta^{'}(0)}^{2}}.
\end{equation}
Now we calculate $D^{'}(0)$.
\begin{equation}\mylabel{dcal1}
D^{'}(w)=B\left(\theta^{'}(w-v)\theta(w+v)+\theta(w-v)\theta^{'}(w+v)\right)
-e^{2\pi iv}Aq^{1/4}(q)_{\infty}^{2}E^{'}(w).
\end{equation}
Using (\ref{jtheta}) and (\ref{wwa}), we have
\begin{align}\mylabel{e}
E(w)&=(-1)^{m-2}q^{\frac{m-2}{4}}(q)_{\infty}^{2(m-2)}\frac{e^{-\pi iw(2m+1)}\theta_{2}\theta_{3}\theta_{4}}{2\theta_{2}(w)\theta_{3}(w)\theta_{4}(w)}\displaystyle\prod_{3\leq j\leq m}\frac{1}{\theta(w-a_{j})\theta(w+a_{j})}\nonumber\\
&\quad\times\sum_{n=-\infty}^{\infty}(-1)^{n}q^{(2m+1)n(n+1)/2}\left(\frac{e^{-2\pi iw(2m+1)n}}{1-e^{2\pi i(u-w)}q^{n}}+\frac{e^{2\pi iw(2m+1)(n+1)}}{1-e^{2\pi i(u+w)}q^n}\right).
\end{align}
Differentiating both sides with respect to $w$ and simplifying, we obtain
\begin{align}\mylabel{ep}
E^{'}(w)&=\frac{1}{2}(-1)^{m-2}q^{\frac{m-2}{4}}(q)_{\infty}^{2(m-2)}e^{-\pi iw(2m+1)}\theta_{2}\theta_{3}\theta_{4}\nonumber\\
&\quad\times\bigg\{-\frac{\pi i(2m+1)+\displaystyle\frac{\theta_{2}^{'}(w)}{\theta_{2}(w)}+\displaystyle\frac{\theta_{3}^{'}(w)}{\theta_{3}(w)}
+\displaystyle\frac{\theta_{3}^{'}(w)}{\theta_{3}(w)}+\displaystyle\sum_{3\leq j\leq m}\left(\frac{\theta^{'}(w-a_{j})}{\theta(w-a_{j})}+\frac{\theta^{'}(w+a_{j})}{\theta(w+a_{j})}\right)}{\theta_{2}(w)\theta_{3}(w)\theta_{4}(w)\displaystyle\prod_{3\leq j\leq m}\theta(w-a_{j})\theta(w+a_{j})}\nonumber\\
&\quad\times\sum_{n=-\infty}^{\infty}(-1)^{n}q^{(2m+1)n(n+1)/2}\left(\frac{e^{-2\pi iw(2m+1)n}}{1-e^{2\pi i(u-w)}q^{n}}+\frac{e^{2\pi iw(2m+1)(n+1)}}{1-e^{2\pi i(u+w)}q^n}\right)\nonumber\\
&\quad+\frac{1}{\theta_{2}(w)\theta_{3}(w)\theta_{4}(w)}\displaystyle\prod_{3\leq j\leq m}\frac{1}{\theta(w-a_{j})\theta(w+a_{j})}F^{'}(w)\bigg\},
\end{align}
where
\begin{equation}\mylabel{f}
F(w):=F(w,u,m;q):=\sum_{n=-\infty}^{\infty}(-1)^{n}q^{(2m+1)n(n+1)/2}\left(\frac{e^{-2\pi iw(2m+1)n}}{1-e^{2\pi i(u-w)}q^{n}}+\frac{e^{2\pi iw(2m+1)(n+1)}}{1-e^{2\pi i(u+w)}q^n}\right).
\end{equation}
It is straightforward to see that 
\begin{equation}\mylabel{fp}
F^{'}(0)=2\pi i(2m+1)\sum_{n=-\infty}^{\infty}\frac{(-1)^{n}q^{(2m+1)n(n+1)/2}}{1-e^{2\pi iu}q^{n}}.
\end{equation}
From (\ref{todd}), we have
\begin{equation}\mylabel{tpodd}
\theta^{'}(-z)=\theta^{'}(z).
\end{equation}
Then letting $w\to 0$ in (\ref{ep}), and using (\ref{todd}), (\ref{tpodd}), (\ref{fp}) and the fact that $\theta_{k}^{'}(0)=0$ for $2\leq k\leq 4$, we find that 
\begin{equation}\mylabel{epr}
E^{'}(0)=0.
\end{equation}
Using (\ref{tpodd}) and (\ref{epr}) in (\ref{dcal1}), we finally deduce that $D^{'}(0)=0$.

%From (\ref{d}), it is straightforward (albeit tedious) to see that $D^{'}(0)=0$, which indeed
With the help of (\ref{mc2}), this then implies that $\lim_{w\to 0}w\left(f_{1}(w)+f_{3}(w)\right)=0$ and thus $\lim_{w\to 0}wg(w)=0$. Thus $w=0$ is also a removable singularity, which implies that $g$ is an doubly periodic entire function and hence a constant, say $K$ (which may very well depend on $v$). Finally, since $J(0)=0$, we have
\begin{equation*}
K=g(v)=\sum_{n=-\infty}^{\infty}(-1)^{n}q^{(2m+1)n(n+1)/2}\left(\frac{e^{-2\pi iv(2m+1)n}}{1-e^{2\pi i(u-v)}q^{n}}+\frac{e^{2\pi iv(2m+1)(n+1)}}{1-e^{2\pi i(u+v)}q^n}\right).
\end{equation*}
This completes the proof.

%%SECTION 3%%%%%%%%%%%%%%%%%%%%%%%%%%%%%%%%%%%%%%%%%%%%%%%%%%%%%%%%%%%%%%%%
\section{Proof of Theorem \thm{RkSig}}
\mylabel{sec:RkSigproof} %% Section 3
From \cite[Eq.(4.3)]{gmm}, we see that
{\allowdisplaybreaks\begin{align}\mylabel{2nd}
R_{k}(z, q)
&=\frac{1}{(q)_{\infty}}\sum_{n=1}^{\infty}(-1)^{n-1}q^{n((2k-1)n-1)/2}(1-q^{n})\left(\frac{1}{1-zq^{n}}+\frac{z^{-1}q^{n}}{1-z^{-1}q^{n}}\right)\nonumber\\
&=\frac{z^{-1}}{(q)_{\infty}}\sum_{n=-\infty \atop{n\neq 0}}^{\infty}(-1)^{n-1}q^{n((2k-1)n+1)/2}\frac{1-q^{n}}{1-z^{-1}q^{n}}.\nonumber\\
\end{align}
Replacing $z$ by $z^{-1}$ in (\ref{r3gfk}) and (\ref{2nd}),
%% and using the Jacobi triple product identity in the third equality below, 
we see that
%%\begin{align*}
%%&R^{(k)}(z, q)\nonumber\\
%%&=\frac{z}{(q)_{\infty}}\sum_{n=-\infty \atop{n\neq 0}}^{\infty}(-1)^{n-1}q^{n((2k-1)n+1)/2}\frac{1-q^{n}}{1-zq^{n}}\nonumber\\
%%&=\frac{z}{(q)_{\infty}}\left(\sum_{n=-\infty \atop{n\neq 0}}^{\infty}\frac{(-1)^{n-1}q^{n((2k-1)n+1)/2}}{1-zq^{n}}+\sum_{n=-\infty \atop{n\neq 0}}^{\infty}\frac{(-1)^{n}q^{n((2k-1)n+3)/2}}{1-zq^{n}}\right)\nonumber\\
%%&=\frac{z}{(q)_{\infty}}\bigg(1-(q)_{\infty}\prod_{n=1\atop {n\neq 0, \pm(k-1) (\text{mod} (2k-1))}}^{\infty}\frac{1}{1-q^{n}}+\sum_{n=-\infty \atop{n\neq 0}}^{\infty}(-1)^{n}q^{n((2k-1)n+1)/2}\nonumber\\
%%&\quad\quad\quad\quad+\sum_{n=-\infty \atop{n\neq 0}}^{\infty}\frac{(-1)^{n-1}q^{n((2k-1)n+1)/2}}{1-zq^{n}}+\sum_{n=-\infty \atop{n\neq 0}}^{\infty}\frac{(-1)^{n}q^{n((2k-1)n+3)/2}}{1-zq^{n}}\bigg)\nonumber\\
%%&=\frac{z}{(q)_{\infty}}\bigg(1-(q)_{\infty}\prod_{n=1\atop {n\neq 0, \pm(k-1) (\text{mod} (2k-1))}}^{\infty}\frac{1}{1-q^{n}}-\sum_{n=-\infty \atop{n\neq 0}}^{\infty}\frac{(-1)^{n}q^{n((2k-1)n+1)/2}(1-(1-zq^{n}))}{1-zq^{n}}\nonumber\\
%%&\quad\quad\quad\quad\quad\quad+\sum_{n=-\infty \atop{n\neq 0}}^{\infty}\frac{(-1)^{n}q^{n((2k-1)n+3)/2}}{1-zq^{n}}\bigg)\nonumber\\
%%&=\frac{-z\theta_{1,2k-1}(q)}{(q)_{\infty}}+\frac{z}{(q)_{\infty}}\bigg(1+(1-z)\sum_{n=-\infty \atop{n\neq 0}}^{\infty}\frac{(-1)^{n}q^{n((2k-1)n+3)/2}}{1-zq^{n}}\bigg)\nonumber\\
\allowdisplaybreaks
\begin{align*}
&R_k(z, q)\nonumber\\
&=\frac{z}{(q)_{\infty}}
\sum_{n=-\infty \atop{n\neq 0}}^{\infty}(-1)^{n-1}q^{n((2k-1)n+1)/2}
\frac{1-q^{n}}{1-zq^{n}}\nonumber
\\
&=\frac{z}{(q)_{\infty}}
\sum_{n=-\infty \atop{n\neq 0}}^{\infty}(-1)^{n-1}q^{n((2k-1)n+1)/2}
\left(1-\frac{(1-z)q^{n}}{1-zq^{n}} \right)\nonumber
\\
&=\frac{z}{(q)_{\infty}}
\left(\sum_{n=-\infty \atop{n\neq 0}}^{\infty} (-1)^{n-1}q^{n((2k-1)n+1)/2} 
+ (1-z)\sum_{n=-\infty \atop{n\neq 0}}^{\infty}\frac{(-1)^{n}q^{n((2k-1)n+3)/2}}
                                                         {1-zq^{n}}\right)\nonumber
\\
&=\frac{z}{(q)_{\infty}}
\left(1- \sum_{n=-\infty}^\infty (-1)^nq^{n((2k-1)n+1)/2} 
%%(q)_{\infty}\prod_{n=1\atop {n\neq 0, \pm(k-1) (\text{mod} (2k-1))}}^{\infty}
%%\frac{1}{1-q^{n}}
+ (1-z)\sum_{n=-\infty \atop{n\neq 0}}^{\infty}
\frac{(-1)^{n}q^{n((2k-1)n+3)/2}}{1-zq^{n}}\right)\nonumber
\\
&=\frac{-z\theta_{1,2k-1}(q)}{(q)_{\infty}}
+\frac{z}{(q)_{\infty}}
\bigg(1+(1-z)\sum_{n=-\infty \atop{n\neq 0}}^{\infty}
\frac{(-1)^{n}q^{n((2k-1)n+3)/2}}{1-zq^{n}}\bigg)\nonumber\\
&=\frac{-z\theta_{1,2k-1}(q)}{(q)_{\infty}}+\frac{z}{(q)_{\infty}}\bigg(1+(1-z)\sum_{n=-\infty \atop{n\neq 0}}^{\infty}(-1)^{n}q^{n((2k-1)n+3)/2}\left(\frac{z^{k-2}q^{(k-2)n}}{1-zq^{n}}+\frac{1-(zq^{n})^{k-2}}{1-zq^{n}}\right)\bigg)\nonumber\\
&=\frac{-z\theta_{1,2k-1}(q)}{(q)_{\infty}}+\frac{z}{(q)_{\infty}}\bigg(1+(1-z)\sum_{n=-\infty \atop{n\neq 0}}^{\infty}(-1)^{n}q^{n((2k-1)n+3)/2}\left(\frac{z^{k-2}q^{(k-2)n}}{1-zq^{n}}+\sum_{m=0}^{k-3}z^mq^{mn}\right)\bigg)\nonumber\\
&=\frac{-z\theta_{1,2k-1}(q)}{(q)_{\infty}}+\frac{z}{(q)_{\infty}}\bigg(1+z^{k-2}(1-z)\sum_{n=-\infty \atop{n\neq 0}}^{\infty}\frac{(-1)^{n}q^{(2k-1)n(n+1)/2}}{1-zq^{n}}\nonumber\\
&\quad\quad\quad\quad\quad\quad\quad\quad\quad\quad\quad+(1-z)\sum_{m=0}^{k-3}z^{m}\sum_{n=-\infty \atop{n\neq 0}}^{\infty}(-1)^{n}q^{n((2k-1)n+2m+3)/2}\bigg)\nonumber\\
&=\frac{-z\theta_{1,2k-1}(q)}{(q)_{\infty}}+\frac{z}{(q)_{\infty}}\bigg(z^{k-2}(1-z)\sum_{n=-\infty}^{\infty}\frac{(-1)^{n}q^{(2k-1)n(n+1)/2}}{1-zq^{n}}\nonumber\\
&\quad\quad\quad\quad\quad\quad\quad\quad\quad\quad\quad+(1-z)\sum_{m=0}^{k-3}z^{m}\sum_{n=-\infty}^{\infty}(-1)^{n}q^{n((2k-1)n+2m+3)/2}\bigg)\nonumber\\
&=\frac{1}{(q)_{\infty}}\left(-z\theta_{1,2k-1}(q)+z^{k-1}(1-z)\Sigma^{(k)}(z,q)+z(1-z)\sum_{m=0}^{k-3}z^{m}\theta_{2m+3,2k-1}(q)\right).
\end{align*}}
This completes the proof of Theorem \thm{RkSig}.
\section{Higher order Rank-Crank-type PDEs}
\mylabel{sec:higher} %% Section 4
In this section we show 
how the generalized Lambert series identity \eqn{shc} can be used to derive general 
Rank-Crank PDEs 
of the type found by Zwegers.
%%%%%%%%%%%%%%%%%%%%%%%%%%%%%%%%%%%%%%%%%
\subsection{The idea}                        
\mylabel{subsec:idea} %% Section 4.1
First we let $x_{i}=\zeta^{i}, 1\leq i\leq m$ in \eqn{shc} to obtain   
%%{\allowdisplaybreaks\begin{align}\mylabel{eq:shcspl}
%%&\zeta^{m}\frac{\left[\zeta^{m+1}\right]_{\infty}}{\left[\zeta^{m}\right]_{\infty}}\sum_{n=-\infty}^{\infty}\frac{(-1)^nq^{(2m+1)n(n+1)/2}}{1-zq^n}\nonumber\\
%%&+\frac{\left[\zeta^{-(m-1)}, \zeta^{-(m-2)}, \cdots ,\zeta^{-2}, \zeta^{-1}, \zeta, \zeta^{2}, \cdots, \zeta^{m},  \zeta^{m+1}\right]_{\infty}(q)_{\infty}^{2}}{\left[z\zeta^{-m}, z\zeta^{-(m-1)}, \cdots ,z\zeta^{-2}, z\zeta^{-1}, z, z\zeta, z\zeta^{2}, \cdots, z\zeta^{m-1}, z\zeta^{m}\right]_{\infty}}\nonumber\\
%%&-\sum_{j=1}^{m-1}\frac{\left[\zeta^{-(m-1)}, \zeta^{-(m-2)}, \cdots, \zeta^{-(m-j)}\right]_{\infty}}{\left[\zeta^{m+2}, \cdots, \zeta^{m+j+1}\right]_{\infty}}\nonumber\\
%%&\quad\quad\quad\times\sum_{n=-\infty}^{\infty}(-1)^{n}q^{(2m+1)n(n+1)/2}\left(\frac{\zeta^{-(j+1)(2m+1)n}}{1-z\zeta^{-(j+1)}q^{n}}+\frac{\zeta^{(j+1)(2m+1)(n+1)}}{1-z\zeta^{j+1}q^n}\right)\nonumber\\
%%&=\sum_{n=-\infty}^{\infty}(-1)^{n}q^{(2m+1)n(n+1)/2}\left(\frac{\zeta^{-(2m+1)n}}{1-z\zeta^{-1}q^{n}}+\frac{\zeta^{(2m+1)(n+1)}}{1-z\zeta q^n}\right).
%%\end{align}}
\begin{equation}
 Y_m(\zeta,z,q) \, (q)_\infty^2
= S_{2m+1}(\zeta,z,q)
+ \sum_{j=1}^{m-1} F_{j,m}(\zeta,q) \, S_{2m+1}(\zeta^{j+1},z,q) 
-F_{0,m}(\zeta,q)\,\Sigma^{(2m+1)}(z,q),
\mylabel{eq:ChanID}
\end{equation}
where
\begin{equation}
S_k(\zeta,z,q)
:=\sum_{n=-\infty}^{\infty}(-1)^{n}
q^{kn(n+1)/2}
\left(\frac{\zeta^{-kn}}{1-z\zeta^{-1}q^{n}}
+\frac{\zeta^{k(n+1)}}{1-z\zeta q^n}\right),
\mylabel{eq:Skdef}
\end{equation}
for $k$ odd and
\begin{align}
F_{0,m}(\zeta,q) &:=\zeta^{m}\frac{\left[\zeta^{m+1}\right]_{\infty}}
                                 {\left[\zeta^{m}\right]_{\infty}}, 
\mylabel{eq:F0mdef}\\
F_{j,m}(\zeta,q) &:=
\frac{\left[\zeta^{-(m-1)}, \zeta^{-(m-2)}, \cdots, \zeta^{-(m-j)}\right]_{\infty}}
     {\left[\zeta^{m+2}, \cdots, \zeta^{m+j+1}\right]_{\infty}} 
     \quad\mbox{(for $1\le j \le m-1$)},
\mylabel{eq:Fjmdef}\\
Y_m(\zeta,z,q)   &:=
\frac{\left[\zeta^{-(m-1)}, \zeta^{-(m-2)}, \cdots ,\zeta^{-2}, \zeta^{-1}, \zeta, 
              \zeta^{2}, \cdots, \zeta^{m},  \zeta^{m+1}\right]_{\infty}}
     {\left[z\zeta^{-m}, z\zeta^{-(m-1)}, \cdots ,z\zeta^{-2}, z\zeta^{-1}, 
        z, z\zeta, z\zeta^{2}, \cdots, z\zeta^{m-1}, z\zeta^{m}\right]_{\infty}}.
\mylabel{Ymdef}
\end{align}
%%SYMBOLPDEGEN:=proc(k)
%%     m:=1/2*(k-1):
%%     LHS:=CKCOF(m)*CS^k*QINF^k:
%%     RHS1:=NEWPOLY(_H[k],k,2*m):
%%     RHS2:=add(add(binomial(2*m,a)*DaFj(a,m,j)*(j+1)^(2*m-a)*NEWPOLY(_H[k],k,2*m-a),
%%           a=0..2*m),j=1..m-1):
%%     RHS3:=-DaF1(2*m,m):
%%     LHS=collect(expand(RHS1+RHS2+RHS3),_H[k]):
%%end:
The basic idea is to apply the operator $D_{2m}$ to both sides of \eqn{ChanID} where
\begin{equation}
D_{\ell} := 
\left.\left(\zeta\frac{\partial}{\partial \zeta}\right)^\ell \right|_{\zeta=1}
=\left.\delta_\zeta^\ell \right|_{\zeta=1}.
\mylabel{eq:Delldef}
\end{equation}
We will also need the differential operator
\begin{equation}
\mathcal{H}^*_k := k\delta_z + 2k \delta_q + \delta_z^2. 
\mylabel{eq:Hstarkdef}
\end{equation}
We note that the operator $\mathcal{H}^*_k$ differs from Zwegers's
$\mathcal{H}_k$ although they are similar.
%%
%%
%%
%%For $k$ odd we define
%%\begin{equation}
%%S_k(\zeta,z,q)
%%:=\sum_{n=-\infty}^{\infty}(-1)^{n}
%%q^{kn(n+1)/2}
%%\left(\frac{\zeta^{-kn}}{1-z\zeta^{-1}q^{n}}
%%+\frac{\zeta^{k(n+1)}}{1-z\zeta q^n}\right).
%%\label{Skdef}
%%\end{equation}
%%Note that $S_{2m+1}(\zeta,z,q)$ is the function that occurs on the right side 
%%of \eqn{shcspl}.
First we need to write the functions $\Sigma^{(k)}(z,q)$ and $S_k(z,q)$
as double series. Throughout we assume that $0<\abs{q}<1$, $z\not\in
\{q^n\,:\,n\in\mathbb{Z}\} \cup \{0\}$ and 
$\zeta\not\in\{z^{\pm1}q^n\,:\,n\in\mathbb{Z}\} \cup \{0\}$. We obtain
\begin{align}
S_k(\zeta,z,q)
&=\sum_{n=0}^{\infty}(-1)^{n}
q^{kn(n+1)/2}
\left(\zeta^{-kn}\sum_{m=0}^\infty z^m \zeta^{-m} q^{mn}
+\zeta^{k(n+1)}\sum_{m=0}^\infty z^m \zeta^{m} q^{mn}\right) \nonumber\\
& - \sum_{n=1}^{\infty}(-1)^{n}
q^{kn(n-1)/2}
\left(\zeta^{kn}\sum_{m=1}^\infty z^{-m} \zeta^{m} q^{mn}
+\zeta^{k(-n+1)}\sum_{m=1}^\infty z^{-m} \zeta^{-m} q^{mn}\right)\nonumber\\
&=\sum_{n=0}^{\infty}\sum_{m=0}^\infty
   (-1)^n z^m q^{kn(n+1)/2 + mn} \left(\zeta^{-kn-m} + \zeta^{k(n+1)+m}\right)
\nonumber\\
&-\sum_{n=1}^{\infty}\sum_{m=1}^\infty
   (-1)^n z^{-m} q^{kn(n-1)/2 + mn} \left(\zeta^{kn+m} + \zeta^{-kn+k-m}\right)
\mylabel{eq:Skid}
\end{align}
and
\begin{equation}
\Sigma^{(k)}(z,q)
=\sum_{n=0}^{\infty}\sum_{m=0}^\infty
   (-1)^n z^m q^{kn(n+1)/2 + mn} 
-\sum_{n=1}^{\infty}\sum_{m=1}^\infty
   (-1)^n z^{-m} q^{kn(n-1)/2 + mn}. 
\mylabel{eq:Sigid}
\end{equation}
%%Now define the differential operators
%%\begin{align*}
%%D_{m} := 
%%\left.\left(\zeta\frac{\partial}{\partial \zeta}\right)^m \right|_{\zeta=1}
%%=\left.\delta_\zeta^m \right|_{\zeta=1},
%%\\
%%\mathcal{H}^*_k := k\delta_z + 2k \delta_q + \delta_z^2. 
%%\end{align*}
We have
\begin{theorem}\mylabel{thm:DellSkIDTHM}
Suppose $k$ is odd and  $1\le \ell \le k-1$. Then
\begin{equation}
\mylabel{eq:DellSkID}
D_{\ell}\, S_k(\zeta,z,q) 
= P_{k,\ell}(\mathcal{H}^*_k) \Sigma^{(k)}(z,q), 
\end{equation}
where
\begin{equation}
P_{k,\ell}(x) = \sum_{m=0}^{\lfloor \ell/2\rfloor}
\frac{\ell(\ell-m-1)!}{(\ell-2m)!m!} x^m k^{\ell-2m}.
\mylabel{Pkelldef}
\end{equation}
\end{theorem}
\begin{proof}
Suppose $k$ is odd and $1\le \ell \le k-1$ . First we prove that
\begin{equation}
\mylabel{Pellid}
P_{k,\ell}(x) =  (\tfrac{1}{2}k - \tfrac{1}{2}\sqrt{k^2+4x})^\ell
              +  (\tfrac{1}{2}k + \tfrac{1}{2}\sqrt{k^2+4x})^\ell.
\end{equation}
We have
\begin{align*}
 & (\tfrac{1}{2}k - \tfrac{1}{2}\sqrt{k^2+4x})^\ell
              +  (\tfrac{1}{2}k + \tfrac{1}{2}\sqrt{k^2+4x})^\ell
=\sum_{j=0}^{\lfloor \ell/2\rfloor} 
 \binom{\ell}{2j} k^{\ell-2j} 2^{1-\ell}(k^2+4x)^j \\
&\quad=\sum_{j=0}^{\lfloor \ell/2\rfloor} 
   \sum_{m=0}^j \binom{\ell}{2j} \binom{j}{m} x^m k^{\ell-2m} 2^{2m-\ell+1} 
=\sum_{m=0}^{\lfloor \ell/2\rfloor} 
   \left(\sum_{j=m}^{\lfloor \ell/2\rfloor} \binom{\ell}{2j} \binom{j}{m} \right)
   x^m k^{\ell-2m} 2^{2m-\ell+1}.
\end{align*}
The result (\ref{Pellid}) now follows from the binomial coefficient
identity
\begin{equation}
\mylabel{binomid}
\sum_{j=m}^{\lfloor \ell/2\rfloor} \binom{\ell}{2j} \binom{j}{m}
= 2^{\ell-2m-1} \frac{\ell(\ell-m-1)!}{(\ell-2m)! m!},
\end{equation}
which we leave as an exercise.

We observe that if $x=km + m^2 + k^2n(n+1) + 2mnk$ then
\begin{align*}
   k^2 + 4x &= (k + 2m + 2kn)^2,\\
\tfrac{1}{2}k - \tfrac{1}{2}\sqrt{k^2+4x} &= -kn -m,\\
\tfrac{1}{2}k + \tfrac{1}{2}\sqrt{k^2+4x} &= k(n+1)+m ,
\end{align*}
and we see that
$$
D_{\ell} \left(\zeta^{-kn-m} + \zeta^{k(n+1)+m}\right)
=
(-kn-m)^\ell + (k(n+1)+m)^\ell
=P_{k,\ell}(km + m^2 + k^2n(n+1) + 2mnk).
$$              

Similarly we find that
$$
D_{\ell} \left(\zeta^{kn+m} + \zeta^{-kn+k-m}\right)
=
(kn+m)^\ell + (-kn+k-m)^\ell
=P_{k,\ell}(-km + m^2 + k^2n(n-1) + 2mnk).
$$              

We note that
\begin{align*}
\mathcal{H}^{*}_k \left( q^{kn(n+1)/2 + mn} z^m\right)
&= (km + m^2 + k^2n(n+1) + 2mnk) \left( q^{kn(n+1)/2 + mn} z^m\right),\\
\mathcal{H}^{*}_k \left( q^{kn(n-1)/2 + mn} z^{-m}\right)
&= (-km + m^2 + k^2n(n-1) + 2mnk) \left( q^{kn(n-1)/2 + mn} z^{-m}\right).
\end{align*}

Thus
$$
D_{\ell} \left(q^{kn(n+1)/2 + mn} z^m(\zeta^{-kn-m} + \zeta^{kn+k+m})\right)
= P_{k,\ell}(\mathcal{H}^{*}_k) \left( q^{kn(n+1)/2 + mn} z^m\right),
$$
and
$$
D_{\ell} \left(q^{kn(n+1)/2 + mn} z^m(\zeta^{kn+m} + \zeta^{-kn+k-m})\right)
= P_{k,\ell}(\mathcal{H}^{*}_k) \left( q^{kn(n-1)/2 + mn} z^{-m}\right).
$$
The result \eqn{DellSkID}   follows from equations \eqn{Skid}   
and \eqn{Sigid}.   
\end{proof}

Next we calculate $D_{2m}$ of each term in \eqn{ChanID}.

%%%%%%%%%%%%%%%%%%%%%%%%%%%%%%%%%%%%%%%%%
\subsection{The term $Y_m(\zeta,z,q)$}
\mylabel{subsec:Ymterm} %% Section 4.2
It is clear that the function $Y_m(\zeta,z,q)$ has a zero of order $2m$ at
$\zeta=1$. It is well-known that
$$
D_{2m}(f(\zeta)) = \sum_{i=1}^{2m} S(2m,j) f^{(j)}(1),
$$
where the numbers $S(2m,j)$ are Stirling numbers of the second kind.
Since $S(2m,2m)=1$ it follows that
\begin{equation}
D_{2m}\left(Y_m(\zeta,z,q))\right) = Y_m^{(2m)}(1,z,q)
= (-1)^{m-1} (2m)!\,(m+1)!\,(m-1)! \, [C^*(z,q)]^{2m+1} (q)_\infty^{2m-1}
\mylabel{eq:Ymterm}
\end{equation}
by an easy calculation.

%%%%%%%%%%%%%%%%%%%%%%%%%%%%%%%%%%%%%%%%%
\subsection{The term $F_{0,m}(\zeta,q)$}
\mylabel{subsec:F0mterm} %% Section 4.3
By logarithmic differentiation we have
\begin{equation}
\delta_\zeta F_{0,m}(\zeta,q) = L_{0,m}(\zeta,q) \, F_{0,m}(\zeta,q).
\mylabel{eq:dzetaF0}
\end{equation}
where
\begin{align}
L_{0,m}(\zeta,q) 
&= K_{0,m}(\zeta) 
- (m+1) \sum_{i=1}^\infty \left(
  \frac{\zeta^{m+1}q^i}{1-\zeta^{m+1}q^i} - \frac{\zeta^{-m-1}q^i}{1-\zeta^{-m-1}q^i}\right)
\nonumber\\
& \qquad\qquad + m \sum_{i=1}^\infty \left(
  \frac{\zeta^{m}q^i}{1-\zeta^{m}q^i} - \frac{\zeta^{-m}q^i}{1-\zeta^{-m}q^i}\right)
\nonumber\\
&= K_{0,m}(\zeta) 
- (m+1) \sum_{i,n\ge1} (\zeta^{m+1}q^i)^n - (\zeta^{-m-1}q^i)^n 
+ m \sum_{i,n\ge1} (\zeta^{m}q^i)^n - (\zeta^{-m}q^i)^n 
\mylabel{eq:Lomdef}
\\
K_{0,m}(\zeta)
&= m + J_{m}(\zeta) - J_{m-1}(\zeta),  
\mylabel{eq:Gomdef}\\
J_m(\zeta)
&= \frac{\displaystyle\sum_{n=1}^m n\zeta^n}{\displaystyle\sum_{n=0}^m \zeta^n}.
\mylabel{eq:Jmdef}
\end{align}
For any positive integer $k$ we define
\begin{equation}
G_{2k} := G_{2k}(q) 
:= \frac{1}{2}\zeta(1-2k) + \sum_{n=1}^\infty \frac{n^{2k-1}q^n}{1-q^n}
= - \frac{B_{2k}}{4k} + \Phi_{2k-1}(q),                             
\mylabel{eq:GEis}
\end{equation}
where $B_{2n}$ is the $(2n)$-th Bernoulli number, and
\begin{equation}
\Phi_{2k-1} := \Phi_{2k-1}(q) := 
\sum_{n=1}^\infty \sigma_{2k-1}(n) q^n.
\mylabel{eq:Phidef}
\end{equation}
The function $G_{2k}$ is a normalized Eisenstein series. 
For $k>1$ it is an entire modular form of weight $2k$.
We need the following
\begin{lemma}
\mylabel{lem:DaJm}
If $m$ and $a$ are positive integers, then
$$
D_a(J_m(\zeta)) = \frac{B_{a+1}}{a+1}\left((m+1)^{a+1}-1\right).
$$
\end{lemma}
\begin{proof}
Suppose $m$ and $a$ are positive integers. It is well-known that
\begin{equation}
\frac{x}{e^x-1} = \sum_{k=0}^\infty \frac{B_k x^k}{k!}.
\mylabel{eq:Berngen}
\end{equation}
By taking the logarithmic derivative of $(\zeta^{m+1}-1)/(\zeta-1)$ we find 
that
$$
J_m(\zeta) = m + (m+1) \frac{1}{\zeta^{m+1}-1} - \frac{1}{\zeta-1}.
$$
Hence by \eqn{Berngen} we have
$$
J_m(e^x) = m + \sum_{k=0}^\infty \frac{B_{k+1}}{(k+1)!}
                          \left( (m+1)^{k+1}-1 \right) x^k.
$$
The result now follows since
$$
D_{a}(J_m(\zeta)) = \left.\left( \frac{d}{dx} \right)^a J_m(e^x) \right|_{x=0}.
$$
\end{proof}
\begin{corollary}
\mylabel{cor:DaL0m}
Suppose $a$, $m$ are integers $a\ge0$ and $m\ge1$. Then
\begin{equation}
D_a(L_{0,m}(\zeta,q))
= 
%%\begin{cases} %old version
%%    g_{0,m,a} + 2 (m^{a+1}-(m+1)^{a+1}) \Phi_a(q)  & \mbox{if $a$ is odd}\\
%%    0 & \mbox{if $a$ is positive and even}\\
%%    m+\frac{1}{2} & \mbox{if $a=0$},
%%\end{cases}
\begin{cases} %old version
    2 (m^{a+1}-(m+1)^{a+1}) G_{a+1}(q)  & \mbox{if $a$ is odd},\\
    m+\frac{1}{2} & \mbox{if $a=0$},\\
    0 & \mbox{otherwise.}
\end{cases}
\mylabel{eq:DaL0}
\end{equation}
\end{corollary}
\begin{proof}
The proof of \eqn{DaL0} when $a=0$ is straightforward. 
Suppose $a$ is even and positive. Then by Lemma \lem{DaJm}
\begin{align*}
D_a(L_{0,m}(\zeta,q)) &=
    D_a(K_{0,m}(\zeta)) 
   \\
&=
\frac{B_{a+1}}{a+1}(-m^{a+1}+(m+1)^{a+1}) \\
&=0,
\end{align*}
since the Bernoulli numbers $B_k$ are zero when $k>1$ is odd.
Now suppose $a$ is odd. Then again by Lemma \lem{DaJm}
\begin{align*}
D_a(L_{0,m}(\zeta,q)) &=
    D_a(K_{0,m}(\zeta)) 
              + 2 (m^{a+1}-(m+1)^{a+1}) \Phi_a(q) \\
&=
\frac{B_{a+1}}{a+1}(-m^{a+1}+(m+1)^{a+1})
              + 2 (m^{a+1}-(m+1)^{a+1}) \Phi_a(q) \\
&=
2 (m^{a+1}-(m+1)^{a+1}) G_{a+1}(q).
\end{align*}
\end{proof}
By applying $D_{a-1}$ to both sides of \eqn{dzetaF0} and
using \eqn{DaL0} we obtain the following recurrence
%%\begin{equation}
%%D_a(F_{0,m}(\zeta,q)) = (m + \tfrac{1}{2}) D_{a-1}(F_{0,m}(\zeta,q))
%% + \sum_{i}^{\lfloor a/2 \rfloor} \binom{a-1}{2i-1} D_{2i-1}(L_{0,m}(\zeta,q)) 
%%D_{a-2i}(F_{0,m}(\zeta,q)).
%%\label{rec1}
%%\end{equation}
\begin{align}
D_a(F_{0,m}(\zeta,q)) &= (m + \tfrac{1}{2}) D_{a-1}(F_{0,m}(\zeta,q)) \nonumber\\
& \quad  + \sum_{i=1}^{\lfloor a/2 \rfloor} 2 \binom{a-1}{2i-1} 
    (m^{2i}-(m+1)^{2i}) G_{2i}(q) D_{a-2i}(F_{0,m}(\zeta,q)).
\mylabel{eq:rec1}
\end{align}
This together with the initial value 
\begin{equation}
D_0(F_{0,m}(\zeta,q))=F_{0,m}(1,q) =\tfrac{m+1}{m}
\mylabel{eq:D0F0m}
\end{equation}
uniquely determines the coefficients 
$D_a(F_{0,m}(\zeta,q))$.
%% We compute some examples
%% \begin{align*}
%% D_0(F_{0,2}) & = \frac{3}{2},\\
%% D_1(F_{0,2}) & = \frac{15}{4},\\
%% D_2(F_{0,2}) & = 10 - 15 \Phi_1,\\
%% D_3(F_{0,2}) & = \frac{225}{8} - \frac{225}{2} \Phi_1,\\
%% D_4(F_{0,2}) & = 82 - 600 \Phi_1 + 450 \Phi_1^2 - 195 \Phi_3.
%% \end{align*}
%%expand(DaFomb(0,2));
%%                                      3
%%                                      -
%%                                      2
%%expand(DaFomb(1,2));
%%                                     15
%%                                     --
%%                                     4 
%%expand(DaFomb(2,2));
%%                                75          
%%                                -- - 15 G[2]
%%                                8           
%%expand(eval(subs({G[2]=_G(2),G[4]=_G(4)},%)));
%%                               10 - 15 PHI[1]
%%expand(DaFomb(3,2));
%%                               375   225     
%%                               --- - --- G[2]
%%                               16     2      
%%expand(eval(subs({G[2]=_G(2),G[4]=_G(4)},%)));
%%                              225   225       
%%                              --- - --- PHI[1]
%%                               8     2        
%%expand(DaFomb(4,2));
%%                   1875   1125                2           
%%                   ---- - ---- G[2] + 450 G[2]  - 195 G[4]
%%                    32     2                              
%%expand(eval(subs({G[2]=_G(2),G[4]=_G(4)},%)));
%%                                             2             
%%                 82 - 600 PHI[1] + 450 PHI[1]  - 195 PHI[3]
%%
We compute some examples
\begin{align*}
D_0(F_{0,2}) & = \frac{3}{2},\\
D_1(F_{0,2}) & = \frac{15}{4},\\
D_2(F_{0,2}) & = \frac{75}{8} - 15 G_2 = 10 - 15 \Phi_1,\\
D_3(F_{0,2}) & = \frac{365}{16}-\frac{225}{2} G_2= \frac{225}{8} - \frac{225}{2} \Phi_1,\\
D_4(F_{0,2}) & = \frac{1875}{32} - \frac{1125}{2} G_2 + 450 G_2^2 - 195 G_4
= 82 - 600 \Phi_1 + 450 \Phi_1^2 - 195 \Phi_3.
\end{align*}

%% old lemma
%%\begin{lemma}
%%If $a$ is a positive even integer and $m\ge1$ then
%%\begin{equation}
%%D_a\left( J_m(\zeta)\right) = 0,
%%\mylabel{eq:DaJm}
%%\end{equation}
%%where $J_m(\zeta)$ is defined in \eqn{Jmdef}.  
%%\end{lemma}
%%\begin{proof}
%%Suppose $m\ge1$. Since
%%$$
%%J_m(\zeta) = \frac{\zeta}{1-\zeta} - (m+1) \frac{z^{m+1}}{1-z^{m+1}}.
%%$$
%%\begin{equation}
%%J_m(e^x) + J_m(e^{-x}) = m.
%%\mylabel{eq:Jmprop}
%%\end{equation}
%%We have 
%%$$
%%h^{(a)}(0) = D_a( J_m (\zeta)),\qquad \mbox{where $h(x) = J_m(e^x)$.}
%%$$
%%The result \eqn{DaJm} now follows since by \eqn{Jmprop}
%%the function $h(x) - \frac{m}{2} = J_m(e^x) - \frac{m}{2}$ is an odd function
%%of $x$.
%%\end{proof}

%%%%%%%%%%%%%%%%%%%%%%%%%%%%%%%%%%%%%%%%%
\subsection{The terms $F_{j,m}(\zeta,q)$ ($1 \le j \le m-1$)}
\mylabel{subsec:Fjmterms} %% Section 4.4
Suppose $1 \le j \le m-1$. We may obtain a similar recurrence for 
$D_a(F_{j,m}(\zeta,q))$.
This time we find that
\begin{equation}
\delta_\zeta F_{j,m}(\zeta,q) = L_{j,m}(\zeta,q) \, F_{j,m}(\zeta,q),
\mylabel{eq:dzetaFj}
\end{equation}
for some function $L_{j,m}(\zeta,q)$ that satisfies
\begin{equation}
D_a(L_{j,m}(\zeta,q))
= 
%%\begin{cases} %old version
%%    g_{0,m,a} + 2 (m^{a+1}-(m+1)^{a+1}) \Phi_a(q)  & \mbox{if $a$ is odd}\\
%%    0 & \mbox{if $a$ is positive and even}\\
%%    m+\frac{1}{2} & \mbox{if $a=0$},
%%\end{cases}
\begin{cases} %old version
    2 \sum_{i=1}^j \left((m+i+1)^{a+1}-(m-i)^{a+1}\right) G_{a+1}(q) &
    \mbox{if $a$ is odd},\\
    -j\left(m+\frac{1}{2}\right) & \mbox{if $a=0$},\\
    0 & \mbox{otherwise.}
\end{cases}
\mylabel{eq:DaLjm}
\end{equation}
The proof of \eqn{DaLjm} is analogous to the proof of Corollary \corol{DaL0m}.
By applying $D_{a-1}$ to both sides of \eqn{dzetaFj} and
using \eqn{DaLjm} we obtain the following recurrence
\begin{align}
D_a(F_{j,m}(\zeta,q)) &= -j(m + \tfrac{1}{2}) D_{a-1}(F_{j,m}(\zeta,q)) \nonumber\\
& \quad  + 
\sum_{i=1}^{\lfloor a/2 \rfloor} \sum_{k=1}^j 2 \binom{a-1}{2i-1} 
     ((m+k+1)^{2i}-(m-k)^{2i}) G_{2i}(q) D_{a-2i}(F_{j,m}(\zeta,q)).
\mylabel{eq:rec2}
\end{align}
This together with the initial value 
\begin{equation}
D_0(F_{j,m}(\zeta,q))=F_{j,m}(1,q) =(-1)^j \prod_{i=1}^j \frac{m-i}{m+i+1} 
\mylabel{eq:D0Fjm}
\end{equation}
uniquely determines the coefficients 
$D_a(F_{j,m}(\zeta,q))$.
%% 
%%                                   -1       -1
%%                                0, --, "=", --
%%                                   4        4 
%%                                    5       5
%%                                 1, -, "=", -
%%                                    8       8
%%                        25   15                5   15       
%%                   2, - -- - -- __G(2), "=", - - - -- PHI[1]
%%                        16   2                 4   2        
%%                      125   225              25   225       
%%                   3, --- + --- __G(2), "=", -- + --- PHI[1]
%%                      32     4               16    4        
%%                625   1125                    2   255              
%%           4, - --- - ---- __G(2) - 675 __G(2)  - --- __G(4), "=", 
%%                64     4                           2               
%% 
%%             1                          2   255       
%%             - - 225 PHI[1] - 675 PHI[1]  - --- PHI[3]
%%             4                               2        
We compute some examples
\begin{align*}
D_0(F_{1,2}) & = -\frac{1}{4},\\
D_1(F_{1,2}) & = \frac{5}{8},\\
D_2(F_{1,2}) & = -\frac{25}{16}-\frac{15}{2} G_2 =  -\frac{5}{4}-\frac{15}{2} \Phi_1,\\
D_3(F_{1,2}) & = \frac{125}{32}+\frac{225}{4} G_2 = \frac{25}{16}+\frac{225}{4} \Phi_1,\\
D_4(F_{1,2}) & = -\frac{625}{64}-\frac{1125}{4} G_2 - 675 G_2^2 - \frac{255}{2} G_4
= \frac{1}{4} - 225 \Phi_1 - 675 \Phi_1^2 - \frac{255}{2} \Phi_3.
\end{align*}

%%%%%%%%%%%%%%%%%%%%%%%%%%%%%%%%%%%%%%%%%
\subsection{The terms $S_{2m+1}(\zeta^{j+1},z,q)$ ($1 \le j \le m-1$)}
\mylabel{subsec:S2m1terms} %% Section 4.5
Again suppose that $1 \le j \le m-1$. Consider the operator $T_j$ that
operates on a function $f(\zeta)$ by $T_j (f(\zeta))=f(\zeta^j)$.
Then
\begin{equation}
\delta_\zeta \circ T_j = j \, \left(T_j \circ \delta_\zeta\right).
\mylabel{eq:dzetaTj}
\end{equation}
A simple induction argument gives 
\begin{equation}
\delta_\zeta^a \circ T_j = j^a \, \left(T_j \circ \delta_\zeta^a\right),
\mylabel{eq:DzetapowTj}
\end{equation}
and
\begin{equation}
D_a \circ T_j = j^a \, D_a.
\mylabel{eq:DaTj}
\end{equation}
Thus by \eqn{DellSkID} we have
\begin{equation}
D_\ell S_{2m+1}(\zeta^{j+1},z,q) 
= (j+1)^\ell P_{2m+1,\ell}(\mathcal{H}^*_{2m+1}) \Sigma^{(2m+1)}(z,q). 
\mylabel{eq:DellS2m1zj}
\end{equation}

%%%%%%%%%%%%%%%%%%%%%%%%%%%%%%%%%%%%%%%%%
\subsection{The main theorem}
\mylabel{subsec:mainthm} %% Section 4.6
We are now ready to derive our main theorem.
Applying $D_{2m}$ to both sides of \eqn{ChanID} and using \eqn{DellSkID},
\eqn{Ymterm}, \eqn{rec1}, \eqn{D0F0m}, \eqn{rec2}, \eqn{D0Fjm},
and \eqn{DellS2m1zj} we have
\begin{theorem}
\mylabel{thm:maintheorem}
\begin{align}
&(-1)^{m+1} (2m)!\,(m+1)!\,(m-1)! [C^*(z,q)]^{2m+1} (q)_\infty^{2m+1}
\nonumber\\
&=
 \Bigg( P_{2m+1,2m}(\mathcal{H}^{*}_{2m+1})  
\nonumber\\
&\quad +
\sum_{j=1}^{m-1} \sum_{a=0}^{2m} (j+1)^{2m-a} \binom{2m}{a} D_a(F_{j,m}(\zeta,q)) 
P_{2m+1,2m-a}(\mathcal{H}^{*}_{2m+1})  
\nonumber\\
&\quad  - D_{2m}(F_{0,m}(\zeta,q)) \Bigg) \Sigma^{(2m+1)}(z,q) 
\mylabel{eq:mainresult}
\end{align}
where the coefficient functions $D_a(F_{j,m})(\zeta,q))$ ($0 \le j \le m-1$)
are given recursively by \eqn{rec1} and \eqn{rec2}, and their initial values
\eqn{D0F0m} and \eqn{D0Fjm}.
\end{theorem}

For $n\ge0$ let $\mathcal{V}_n$ be the $\mathbb{Q}$-vector space spanned 
by the monomials
$\Phi_1^a\Phi_3^b\Phi_5^c$ with $a+2b+3c=n$. We Define
\begin{equation}
\mathcal{W}_n = \sum_{k=0}^n \mathcal{V}_n;
\end{equation}
i.e., $\mathcal{W}_n$ is the $\mathbb{Q}$-vector space spanned by the monomials
$\Phi_1^a\Phi_3^b\Phi_5^c$ with $0\le a+2b+3c\le n$. We call $\mathcal{W}_n$ the
space of \textit{quasi-modular forms} of weight less than or equal to $2n$.
This agrees with the definition in \cite[p.355]{ag} except this time we allow
monomials of weight $0$.

\begin{corollary}
\mylabel{cor:maincor}
Suppose $m\ge1$. Then there exist quasi-modular forms $f_j\in \mathcal{W}_j$ 
for 
$1 \le j \le m$ such that
\begin{equation}
 \Bigg( {\mathcal{H}^{*\,m}_{2m+1}} + \sum_{k=0}^{m-1} f_{m-k} 
\,{\mathcal{H}^{*\,k}_{2m+1}}\Bigg)
\Sigma^{(2m+1)}(z,q)
= (2m)!\,[C^*(z,q)]^{2m+1} (q)_\infty^{2m+1}.
\mylabel{eq:mainresultcor}
\end{equation}
\end{corollary}
\begin{proof}
Suppose $m\ge1$. It is well-known that
$$
\Phi_{2n-1} \in \mathcal{W}_n.
$$
See \cite[Eq.(3.25)]{ag}. Equation \eqn{GEis}, the recurrence \eqn{rec2} and
a simple induction argument imply that 
$$
D_a(F_{j,m}(\zeta,q)) \in \mathcal{W}_{\lfloor a/2 \rfloor},
$$
for $1 \le j \le {m-1}$. Similarly using \eqn{DaL0} and \eqn{rec1} we find that
\begin{equation}
D_{2m}(F_{0,m}(\zeta,q)) \in \mathcal{W}_{m}.
\mylabel{eq:D2mF0}
\end{equation}
Now we calculate the coefficient of ${\mathcal{H}^{*\,k}_{2m+1}}$ in the right side
of \eqn{mainresult}. The degree of the polynomial $P_{2m+1,2m-a}(x)$
is $\lfloor (2m-a)/2\rfloor$. Assuming $k \le \lfloor (2m-a)/2\rfloor$ we have 
$2k \le 2m - a$ and $\lfloor a/2 \rfloor \le m-k$, and in this case 
$D_a(F_{j,m}(\zeta,q))$ is in $\mathcal{W}_{m-k}$. This together     with
\eqn{D2mF0} implies that the coefficient of ${\mathcal{H}^{*\,k}_{2m+1}}$ is
in  $\mathcal{W}_{m-k}$  for $0 \le k \le m$. %% Finally the coefficient
The coefficient
of ${\mathcal{H}^{*\,m}_{2m+1}}$ is
\begin{align}
f_0 &= 2 + 2 \sum_{j=1}^{m-1} (-1)^j (j+1)^{2m} \prod_{k=1}^j \frac{(m-k)}{m+k+1}
\nonumber
\\
&=
2\sum_{j=0}^{m-1}(-1)^j(j+1)^{2m} \frac{(m-1)!(m+1)!}{(m-j-1)!(m+j+1)!}
\mylabel{eq:f0}
\\
&=\frac{2(m-1)!(m+1)!}{(2m)!}\sum_{j=0}^{m-1}(-1)^j(j+1)^{2m} \binom{2m}{m-j-1}.
\nonumber
\end{align}
%%
%%= (-1)^{m+1} \, (m+1)! \, (m-1)!. 
%%\mylabel{eq:funnysum}
We show that
\begin{equation}
f_0 = (-1)^{m+1} \, (m+1)! \, (m-1)!. 
\mylabel{eq:f0id}
\end{equation}
In view of \eqn{f0} this is equivalent to showing that
\begin{equation}
\frac{2}{(2m)!}\sum_{j=0}^{m-1}(-1)^{m+j+1}(j+1)^{2m} \binom{2m}{m-j-1}
=1,
\mylabel{eq:f0id1}
\end{equation}
which we can rewrite as
\begin{equation}
2\sum_{j=0}^{m-1}(-1)^{j}(m-j)^{2m} \binom{2m}{j}
=(2m)!,
\mylabel{eq:f0id2}
\end{equation}
by replacing $j$ by $m-j-1$ in the sum.

Since 
$$
\sum_{j=0}^{m-1}(-1)^{j}(m-j)^{2m} \binom{2m}{j} =
\sum_{j=m+1}^{2m}(-1)^{j}(m-j)^{2m} \binom{2m}{j},
$$
it suffices to prove
\begin{equation}
\sum_{j=0}^{2m}(-1)^{j}(m-j)^{2m} \binom{2m}{j} =
(2m)!.
\mylabel{eq:f0id3}
\end{equation}

Beginning with the elementary identity
$$
\sum_{j=0}^{2m}\binom{2m}{j}x^j y^{2m-j} = 
(x+y)^{2m},
$$
setting $x = -\frac{1}{\sqrt{\zeta}}$ and $y = \sqrt{\zeta}$, we obtain
\begin{equation} 
\sum_{j=0}^{2m}\binom{2m}{j}(-1)^j \zeta^{m-j} =
\zeta^{-m}(\zeta-1)^{2m}.
\mylabel{eq:elID}
\end{equation}
We apply $D_{2m}$ to both sides of \eqn{elID} and argue as in Section
\subsect{Ymterm} to obtain \eqn{f0id3} which completes the proof
of \eqn{f0id}.
%% Applying $D_{2m}$ on both sides of \eqn{f03}, noting that on the right side, we recall our previous argument, that ``the only term that remains is the term where $(\zeta-1)^{2m}$ is differentiated $2m$ times", we obtain \eqref{p42_1}.
%% 
%% \bigskip
%% Remark: It seems a more general form of \eqref{p42_1} holds: For any $K\in \mathbb{Z}$, 
%% $$
%% \sum_{j=0}^{2m}(-1)^{j}(Km\pm j)^{2m} \binom{2m}{j} =
%% (2m)!.
%% $$
%% %%
The final result \eqn{mainresultcor} follows 
from \eqn{mainresult} by dividing both sides by $f_0$ and using \eqn{f0id}.
%%
%%LID:=m->2*add( (j+1)^(2*m)*mul( (m-k)/(m+k+1),k=1..j)*(-1)^j,j=1..m-1)+2;
%%          /       (2 m)    /  m - k              \     j                \    
%%m -> 2 add|(j + 1)      mul|---------, k = 1 .. j| (-1) , j = 1 .. m - 1| + 2
%%          \                \m + k + 1            /                      /    
%%RID:=m->(-1)^(m-1)*(m+1)!*(m-1)!;
%%         (m - 1)                                  
%%m -> (-1)        factorial(m + 1) factorial(m - 1)
%%LID(9);
%%                                146313216000
%%RID(9);
%%                                146313216000
%%LID(10);
%%                               -14485008384000
%%RID(10);
%%                               -14485008384000
%%seq(LID(m)-RID(m),m=1..10);
%%                        0, 0, 0, 0, 0, 0, 0, 0, 0, 0
%%
%%
\end{proof}

%%%%%%%%%%%%%%%%%%%%%%%%%%%%%%%%%%%%%%%%%
\subsection{Some examples}
\mylabel{subsec:examples} %% Section 4.7
We illustrate Theorem \thm{maintheorem} and Corollary \corol{maincor} 
with some examples.
We show details of the calculations for the cases $m=1$, $2$. 
In cases $m=3$, $4$
we give the quasi-modular forms 
$f_j$ ($1 \le j \le m$) in Corollary \corol{maincor}, in terms of
the functions $\Phi_{2k-1}(q)$ rather than the $G_{2k}(q)$.

\begin{example}
$m=1$.
\begin{align*}
4\, [C^*(z,q)]^{3} (q)_\infty^{3} 
&= \Bigg(9 + 2\,\mathcal{H}_{3} - D_{2}(F_{0,1}(\zeta,q))\Bigg) \Sigma^{(3)}(z,q)\\
&= \Bigg(9 + 2\,\mathcal{H}_{3} -(5 - 12\,\Phi_1)\Bigg) \Sigma^{(3)}(z,q)\\
&=  \Bigg(2\,\mathcal{H}_{3} + 4 +  12\,\Phi_1\Bigg) \Sigma^{(3)}(z,q),
\end{align*}
and
$$
\Bigg(\mathcal{H}_{3} + 2 +  6\,\Phi_1\Bigg) \Sigma^{(3)}(z,q)=
2\, [C^*(z,q)]^{3} (q)_\infty^{3}.
$$
%%SYMBOLPDEGEN2(3);
%%                       3     3                            
%%                   4 CS  QINF  = 9 + 2 _H[3] - _DaF1(2, 1)
%%SYMBOLPDEGEN3(3);
%%                     3     3                                
%%                 4 CS  QINF  = 9 + 2 _H[3] - {5 - 12 PHI[1]}
%%SYMBOLPDEGEN(3);
%%                        3     3                          
%%                    4 CS  QINF  = 4 + 2 _H[3] + 12 PHI[1]
%%lhs(%)/CKCOF2(1)=collect(expand(rhs(%)/CKCOF2(1)),_H[3]);
%%                         3     3                       
%%                     2 CS  QINF  = 2 + _H[3] + 6 PHI[1]
%%
This identity implies the Rank-Crank PDE \eqn{rcpde} as in 
\cite[Section 2]{ag}.
\end{example}

\begin{example}
$m=2$.
\begin{align*}
&-144\, [C^*(z,q)]^{5} (q)_\infty^{5}\\
&=
\Bigg(625+100\,{\mathcal{H}^{*}}_5+2\,{\mathcal{H}^{*}}_5^2+
16\,D_0(F_{1,2}(\zeta,q))\,(625+100\,{\mathcal{H}^{*}}_5+2\,{\mathcal{H}^{*}}_5^2)\\
&+32\,D_1(F_{1,2}(\zeta,q))\,(125+15\,{\mathcal{H}^{*}}_5)
+24\,D_2(F_{1,2}(\zeta,q))\,(25+2\,{\mathcal{H}^{*}}_5)
+40\,D_3(F_{1,2}(\zeta,q)) \\
&+2\,D_4(F_{1,2}(\zeta,q))
- D_4(F_{0,2}(\zeta,q)) \Bigg) \Sigma^{(5)}(z,q)\\
&=
\Bigg(625+100\,{\mathcal{H}^{*}}_5+2\,{\mathcal{H}^{*}}_5^2
-4\,(625+100\,{\mathcal{H}^{*}}_5+2\,{\mathcal{H}^{*}}_5^2)\\
&+20\,(125+15\,{\mathcal{H}^{*}}_5)
-(30+180\,\Phi_1)(25+2\,{\mathcal{H}^{*}}_5)
%% -(30+180*PHI[1])*(25+2*_H[5])
+ (\frac{125}{2} +2250\,\Phi_1)\\
&+(\frac{1}{2} - 450 \,\Phi_1 - 1350\,\Phi_1^2 - 255\,\Phi_3)
%% +(1/2 - 450 *PHI[1] - 1350*PHI[1]^2 - 255*PHI[3])
+(-82 + 600\,\Phi_1 - 450\Phi_1^2 + 195\,\Phi_3)
%% +(-82 + 600*PHI[1] - 450*PHI[1]^2 + 195*PHI[3]);
\Bigg) \Sigma^{(5)}(z,q),
\end{align*}
%% TPDE5:=625+100*_H[5]+2*_H[5]^2
%% -4*(625+100*_H[5]+2*_H[5]^2)
%% +20*(125+15*_H[5])
%% -(30+180*PHI[1])*(25+2*_H[5])
%% + (125/2 +2250*PHI[1])
%% +(1/2 - 450 *PHI[1] - 1350*PHI[1]^2 - 255*PHI[3])
%% +(-82 + 600*PHI[1] - 450*PHI[1]^2 + 195*PHI[3]);
and
\begin{equation}
\Bigg({\mathcal{H}^{*}}_5^2 + (60\Phi_1+ 10)\,{\mathcal{H}^{*}}_5
  +300\,\Phi_1^2 + 10\,\Phi_3 + 350\,\Phi_1+24
\Bigg) \Sigma^{(5)}(z,q)
=24\, [C^*(z,q)]^{5} (q)_\infty^{5}.
\mylabel{eq:PDEC5}
\end{equation}
In this case of Corollary \corol{maincor} we see that
\begin{align*}
f_1 &= 60\Phi_1 +10,\\
f_2 &= 300\,\Phi_1^2 + 10\,\Phi_3 + 350\,\Phi_1+24.
\end{align*}
We show how this identity implies \eqn{grcpdev2}.
We need the results
$$
\delta_q((q)_\infty) = -\Phi_1\,(q)_\infty,\qquad
\delta_q(\Phi_1) =  \tfrac{1}{6}\Phi_1 - 2\Phi_1^2 + \tfrac{5}{6}\Phi_3.
$$
%% dE3b:=QOPS(EULERPROD^3,2):
%% de3bn:=series(dE3b/EULERPROD^3,q,100):
%% findhomcombo(de3bn,[phi1,phi1^2,phi3],q,1,0,no);
%%                                # of terms , 24
%%             -----possible linear combinations of degree------, 1
%%                         /  1                  5     \ 
%%                        { - - X[1] + 15 X[2] - - X[3] }
%%                         \  2                  2     / 
This implies that
\begin{equation}
\mathcal{H}_5^{*}(\Sigma^{*}(z,q)) 
=\mathcal{H}_5^{*}((q)_\infty^3 G^{(5)}(z,q)) 
=(q)_\infty^3 (\mathcal{H}_5^{*} - 30\Phi_1) G^{(5)}(z,q), 
\mylabel{eq:H5S}
\end{equation}
and
\begin{equation}
\mathcal{H}_5^{*\,2}(\Sigma^{*}(z,q)) 
=\mathcal{H}_5^{*\,2}((q)_\infty^3 G^{(5)}(z,q)) 
=(q)_\infty^3 (\mathcal{H}_5^{*\,2} - 60\Phi_1\mathcal{H}_5^{*}
-50\Phi_1 + 1500\Phi_1^2 - 250\Phi_3) G^{(5)}(z,q). 
\mylabel{eq:H52S}
\end{equation}
Substituting \eqn{H5S}, \eqn{H52S} into \eqn{PDEC5} we find
%%-240*PHI[3]+24+10*H5+H5^2
$$
(\mathcal{H}_5^{*\,2} + 10 \mathcal{H}_5^{*} + 24 - 240\Phi_3) G^{(5)}(z,q)
=24\, [C^*(z,q)]^{5} (q)_\infty^{5},
$$
which simplifies to \eqn{grcpdev2} since
$$
\mathbf{H}_* = \mathcal{H}_5^{*} + 5.
$$
\end{example}

%%24*CS^5*QINF^5 = _H[5]^2+(60*PHI[1]+10)*_H[5]
%%+300*PHI[1]^2+10*PHI[3]+350*PHI[1]+24

\begin{example}
$m=3$.
%%SYMBOLPDEGEN(7);
%%         7     7           3                              2
%% 34560 CS  QINF  = 48 _H[7]  + (10080 PHI[1] + 1344) _H[7] 
%%
%%      /                                                    2\              
%%    + \12096 + 211680 PHI[1] + 10080 PHI[3] + 423360 PHI[1] / _H[7] + 34560
%%
%%    + 1091328 PHI[1] + 117600 PHI[3] + 672 PHI[5] + 141120 PHI[3] PHI[1]
%%
%%                    2                 3
%%    + 4939200 PHI[1]  + 1975680 PHI[1] 
%%PDE7:=lhs(%)/CKCOF2(3)=collect(expand(rhs(%)/CKCOF2(3)),_H[7]);
%%       7     7        3                          2
%% 720 CS  QINF  = _H[7]  + (210 PHI[1] + 28) _H[7] 
%%
%%      /                        2                    \                     3
%%    + \210 PHI[3] + 8820 PHI[1]  + 252 + 4410 PHI[1]/ _H[7] + 41160 PHI[1] 
%%
%%    + 2450 PHI[3] + 14 PHI[5] + 720 + 22736 PHI[1] + 2940 PHI[3] PHI[1]
%%
%%                   2
%%    + 102900 PHI[1] 
%%for j from 1 to 3 do
%%lprint(j,coeff(rhs(PDE7),_H[7],3-j));
%%od;
%%1, 210*PHI[1]+28
%%2, 210*PHI[3]+8820*PHI[1]^2+252+4410*PHI[1]
%%3, 41160*PHI[1]^3+2450*PHI[3]+14*PHI[5]+720+22736*PHI[1]+2940*PHI[3]*PHI[1]+102900*PHI[1]^2
We find that
\begin{align*}
f_1 &= 210\,\Phi_1+28 \\
f_2 &= 210\,\Phi_3+8820\,\Phi_1^2+252+4410\,\Phi_1\\
f_3 &= 41160\,\Phi_1^3+2450\,\Phi_3+14\,\Phi_5+720+22736\,\Phi_1+2940\,\Phi_3\,\Phi_1+102900\,\Phi_1^2
\end{align*}
\end{example}

\begin{example}
$m=4$.
We find that
\begin{align*}
f_1 &= 504\,\Phi_1+60\\
f_2 &= 1260\,\Phi_3+24948\,\Phi_1+1308+68040\,\Phi_1^2\\
f_3 &=  
136080\,\Phi_3\,\Phi_1+504\,\Phi_5+45360\,\Phi_3+2449440\,\Phi_1^2+403704\,\Phi_1+2449440\,\Phi_1^3+12176\\
f_4 &=  
40320+2126232\,\Phi_1+404082\,\Phi_3+9828\,\Phi_5+2653560\,\Phi_3\,\Phi_1+21820428\,\Phi_1^2\\
&\quad +9072\,\Phi_1\,\Phi_5+47764080\,\Phi_1^3+18\,\Phi_7+1224720\,\Phi_3\,\Phi_1^2+11340\,\Phi_3^2+11022480\,\Phi_1^4.
\end{align*}
\end{example}

\section{Concluding remarks}                  
\mylabel{sec:conclusion} %% Section 5

The main goal of this paper was to show how the generalized Lambert series
identity \eqn{shc} leads to the higher level Rank-Crank-type PDEs of Zwegers.
The first author's proof \cite{chan} of \eqn{shc} only involves a partial fraction
argument and this together with the proof in Section \sect{higher} 
gives an elementary $q$-series proof of these
higher level Rank-Crank-type PDEs. The elliptic function proof of \eqn{shc}
in Section \sect{chansidelliptic} is independent of the other sections.   
Our form of Zwegers's result \eqn{ZwegersPDE} was given above in \eqn{mainresultcor}.
In our form the
coefficients are quasimodular forms rather than holomorphic modular
forms. The quasimodular function $E_2$ occurs in Zwegers's result as part of the 
definition of his operator $\mathcal{H}_k$. Our coefficient functions 
are given recursively. It would be interesting to find explicit expressions
for the coefficients and to derive the form of Zwegers's result by our method.
The coefficients in the \thpow{4} order PDE \eqn{grcpdev2} only involve the
holomorphic modular form $E_4$, and the differential operator
$\mathbf{H}_{*}$ does not involve the quasimodular $E_2$. It would be interesting
to determine whether there is a renormalization of higher order Rank-Crank-type 
PDEs
which only involve holomorphic modular forms, either as coefficients
or in the definition of the differential operator.
Bringmann, Lovejoy and Osburn \cite{Br-Lo-Os09}, \cite{Br-Lo-Os10}
found Rank-Crank-type PDEs
for overpartitions. Bringmann and Zwegers \cite{Br-Zw} showed how
these results fit into the framework of non-holomorphic Jacobi forms
and found an infinite family of these PDEs. However these PDEs only involve
Appell functions of level $1$ or $3$. It would be interesting to determine
whether the methods of this paper could be extended to find PDEs for
higher level analogues.

\bigskip

\noindent
\textbf{Acknowledgements}

\noindent
We would like to thank Bruce Berndt and Ken Ono for their comments
and suggestions.

%%HERE

\end{document}